\documentclass{amsart}

\usepackage{graphicx}
\usepackage[all]{xy}
\usepackage{amsfonts}
\usepackage{amsthm}
\usepackage{amssymb}
\usepackage[english]{babel}
\usepackage[utf8]{inputenc}
\usepackage[colorinlistoftodos]{todonotes}
\usepackage{tikz}

\newtheorem{theorem}{Theorem}[section]
\newtheorem{lemma}[theorem]{Lemma}
\newtheorem{prop}[theorem]{Proposition}

\theoremstyle{definition}
\newtheorem{definition}[theorem]{Definition}
\newtheorem{example}[theorem]{Example}

\theoremstyle{remark}
\newtheorem{remark}[theorem]{Remark}

\numberwithin{equation}{section}

\begin{document}

\title{Derived Tame Quadratic String Algebras}

\author{Marlon Pimenta Fonseca}
\address{Departamento de Ci\^encias Exatas, DEX,  Universidade Federal de Lavras}
\curraddr{Av. Central UFLA, S/n - Aquenta Sol,CEP 37200-000, Lavras- MG, Brasil}
\email{marlon.fonseca@ufla.br}


\subjclass[2010]{Primary 16G60; Secondary 16E35}

\date{November 14, 2020.} 


\keywords{Representations of algebras, Derived tame algebras, Quadratic string algebras}

\begin{abstract} In this paper we determine the derived representation type of quadratic string algebras and we prove that every derived tame quadratic string algebra whose quiver has cycles is derived equivalent to some skewed-gentle algebra.
\end{abstract}

\maketitle

\section*{Introduction} Throughout the paper $K$ denotes a fixed algebraically closed field. By an algebra we mean a basic, connected and finite dimensional $K$-algebra (associative, with an identity). For an algebra $A$, let $D^b(A)$ be the bounded derived category of the category of finitely generated right modules $\mathrm{mod}\ A$.

During the last years there has been an active study of derived categories. The notions of derived tameness was introduced in \cite{MR1913081} and the tame-wild dichotomy for bounded derived categories of finite-dimensional algebras was established in \cite{Dicotomy}. In particular, the  study of derived representation type of an algebra becomes an important theme in representation theory. This question is well-known for tree algebras \cite{MR1868357,MR1902063}, for algebras with radical squared zero \cite{MR2497950,MR3646293} and for Nakayama algebras \cite{MR3778425,Nakayama}, for example.    

In this paper we consider the class of quadratic string algebra. Recall that $A= KQ/I$ is a \emph{quadratic string algebra} if the following conditions are satisfied: 
\begin{itemize}
	\item[(i)] each vertex in $Q$ is the source of at most two arrows and the target of at most two arrows;
	\item[(ii)] for each arrow $\alpha \in Q_1$ there is at most one arrow $\beta$ such that $t(\alpha)=s(\beta)$  and $\alpha \beta \notin  I$, and at most one arrow $\gamma$ such that $s(\alpha)=t(\gamma)$ and $\gamma \alpha \notin I$;
	\item[(iii)] $I$ is generated by paths of length two. 
\end{itemize}
This class of algebras is natural in our context since it is a subclass of the special biserial algebras and a generalization of gentle algebras. Moreover, the derived category of quadratic string algebras also were studied by Franco,  Giraldo and Rizzo \cite{SAG} with the name of string almost gentle algebras. 

In the order to classify the derived tame quadratic algebras, we introduce a subclass of quadratic string algebras called \emph{good quadratic string algebras} (shortly  gqs algebras) (see Definition \ref{defi-gqs}). This class is ``good'' because it is defined by combinatorial properties and because every gqs algebra is derived tame.        

Recall that if $A$ is an algebra of finite global dimension, then its Euler quadratic form is defined on the Grothendieck group of $A$ by $$\chi_A(\dim\, M)=\sum_{i=0}^{\infty} (-1)^i\mathrm{dim}_K\mathrm{Ext}_A^i(M, M)$$ for any $A$-module $M$.

Our main results are the following theorems.\\

\noindent\textbf{Theorem A.} Every gqs algebra is derived equivalent to a skewed-gentle algebra. In particular, every qgs algebra is derived tame.\\

\noindent\textbf{Theorem B.} Let $A$ be a quadratic string algebra such that $A$ is not a tree. Then $A$ is derived tame if and only if $A$ is a gqs algebra.\\



Note that by Theorems A and B we have the following result.\\


\noindent\textbf{Corollary C.} Let $A$ be a quadratic string algebra such that $A$ is not a tree. If $A$ is derived tame, then $A$ is derived equivalent to some skewed-gentle algebra.  \\


\noindent\textbf{Theorem D.} Let $A$ be a quadratic string algebra. Then $A$ is derived tame if and only if one of the following conditions holds:
\begin{itemize}
	\item[(1)] $A$ is tree and its Euler form is non-negative;
	\item[(2)] $A$ is a gqs algebra.\\ 
\end{itemize}

Note that, if $A$ is a tree then the Theorem D follows from \cite{MR1868357,MR1902063}.\\

The structure of this paper is as follows. In Section 1 we review some preliminary results about derived representation type classification of algebras. Moreover, we also recall the definitions of Galois coverings, mutations and blowing-up of an algebra. In Section 2 we recall the definition of gentle and skewed-gentle algebras and prove the Theorem A. At last, in the Section 3 we prove the Theorems B  and D.

\section{Preliminaries}
\subsection{Algebras, categories and quivers} Following \cite{MR654725}, a \emph{locally bounded category}  $\mathcal{C}$ is a $K$-linear category satisfying the following conditions: 
\begin{itemize}
	\item[(i)]different objects in $\mathcal{C}$ are not isomorphic;
	\item[(ii)]$\mathcal{C}(x,x)$ is a local ring for every $x \in \mathcal{C}$;
	\item[(iii)]$\dim_K \sum_{y \in \mathcal{C}}\mathcal{C}(x,y) + \dim_K \sum_{y \in \mathcal{C}}\mathcal{C}(y,x)< \infty$ for all $x \in \mathcal{C}$.               
\end{itemize} 
Moreover, if the number of objects in $\mathcal{C}$ is finite, then $\mathcal{C}$ is called a \emph{bounded category}.

For a locally bounded category $\mathcal{C}$, denote by $\mathrm{Mod}\ \mathcal{C}$ the category of all  contravariant functors from $\mathcal{C}$ to the category of $K$-vector spaces, and denote by $\mathrm{mod}\ \mathcal{C}$ the full subcategory of $\mathrm{Mod}\ \mathcal{C}$ formed by the finite-dimensional contravariant functors $M$ (i.e. $\sum_{x \in \mathcal{C}}\dim_K M(x)< \infty$). 

\begin{remark}It is clear that every bounded category $\mathcal{C}$ defines an basic finite-dimensional algebra $ \bigoplus  \mathcal{C}$ formed by quadratic matrices $a = (a_{yx})_{ x,y \in \mathcal{C}}$ such that $a_{yx} \in \mathcal{C}(x, y)$. Conversely, for every basic finite-dimensional algebra $A$ (associative with unity) can be attach to bounded category  $\mathcal{C}_A$ whose objects are a complete set of  primitive orthogonal idempotents of $A$, the space $ yAx$ as set $\mathcal{C}_A(x, y)$ of morphisms from $x$ to $y$ and the composition is induced by the multiplication in $A$. It is easy to see that $A\cong \bigoplus \mathcal{C}_A$ and  the categories $\mathrm{mod}\ \mathcal{C}_A$ and $\mathrm{mod}\ A$ are equivalents. By an abuse of notation, we shall identify an algebra $A$ with its bounded category $\mathcal{C}_A$.
\end{remark}

A \emph{quiver} is a tuple $Q=(Q_0, Q_1, s, t)$, where $Q_0$ and $Q_1$ are sets and $s,t$ are maps $s, t : Q_1 \rightarrow Q_0$. The elements of $Q_0$ and $Q_1$ are called \emph{vertices} and \emph{arrows} of $Q$, respectively, and the \emph{type} of $Q$ is its underlying graph.  We say that $Q$ is finite if $Q_0$ and $Q_1$ are both finite sets, and we say that $Q$ is connected if its underlying graph is a connected graph. In this paper, all the quivers will be considered connected quivers.  For every arrow $\alpha \in Q_1$ the vertex $s(\alpha)$ is its source, while $t(\alpha)$ is its target. An arrow $\alpha \in Q_1$ is called \emph{loop} if $s(\alpha)=t(\alpha)$. A quiver $\Delta$ is \emph{subquiver} of $Q$ if $\Delta_i \subseteq Q_i$ for $i=0,1$. A subquiver $\Delta$ of $Q$ is \emph{full} if $\Delta_1=\{\alpha \in Q_1| s(\alpha), t(\alpha) \in \Delta_0\}$.

For every arrow $\alpha:x \rightarrow y$ of $Q$, we denote by $\alpha^{-1}:y \rightarrow x$ its formal inverse. A \emph{walk} $w$ in $Q$ of length $l(w)=n$  is a sequence $w=\alpha_1^{\epsilon_1}\alpha_2^{\epsilon_2}\cdots \alpha_n^{\epsilon_n}$ of arrows and formal inverses of arrows such that $s(\alpha_{i+1}^{\epsilon_{i+1}})=t(\alpha_{i}^{\epsilon_i})$ for any $1\leq i \leq n-1$. The source and the target of a walk are defined in the natural way. The concatenation $ww'$  of two walks $w$, $w'$ in $Q$ is defined in the natural way whenever $s(w') = t(w)$. A \emph{path} in $Q$ is a walk $ w=\alpha_1\alpha_2\cdots \alpha_n$ constituted only by arrows of $Q$. Moreover, to every vertex $x$, one associates a \emph{trivial path} $e_x$ with $s(e_x)=t(e_x)=x$ which is of length 0 by convention. We denote by $Q_2$ the set of paths of length 2. A walk $w \in Q$ is called \emph{closed} if $s(w)=t(w)$; \emph{reduced} if $w$ is either a trivial path, or $w=\alpha_1^{\epsilon_1}\cdots \alpha_n^{\epsilon_n}$ such that $\alpha_{i+1}^{\epsilon_{i+1}}\not= \alpha_i^{-\epsilon_i}$ for all $1 \leq i \leq n-1$; and a \emph{cycle} if $w$ is non-trivial, reduced and closed.

Given a quiver $Q$, the \emph{path category} $KQ$ is defined as follows: its object class is $Q_0$ and given $x,y \in Q_0$, the morphism set $KQ(x, y)$ is the $K$-vector space having as basis the set of paths from $y$ to $x$ and the composition is induced by concatenation of paths. Denote by $\mathcal{R}_Q$ the two-sided ideal of $KQ$ generated by $Q_1$. An ideal $I$ of $KQ$ is said to be \emph{admissible} if for every $x,y \in Q_0$ one has $I(x,y) \subseteq \mathcal{R}_Q^2(x,y)$ and for any $x \in Q_0$ there exists some $n_x \geq 2$ such that $I$ contains all the paths with length at least $n_x$ and with source or target $x$. If $Q$ is a finite quiver and $I$ is an admissible ideal of $KQ$, the pair $(Q,I)$ is called \emph{bounded quiver}. It is well known that any finite-dimensional $K$-algebra is Morita equivalent to a quotient $KQ/I$ where $I$ is an admissible ideal. Moreover, if $KQ/I$ is a finite-dimensional $K$-algebra, then the \emph{radical} (Jacobson) of $KQ/I$ is given by $\mathrm{rad}\, KQ/I= \mathcal{R}_Q/I$.

A \emph{quiver-morphism} $\phi: Q' \rightarrow Q$ consists of two maps $\phi_0:Q'_0\rightarrow Q_0$ and \linebreak $\phi_1:Q'_1\rightarrow Q_1$ such that for every arrow $\alpha:i \rightarrow j$ of $Q'$, we have that \linebreak $\phi_1(\alpha): \phi_0(i) \rightarrow \phi_0(j)$. Moreover, if $I'$ and $I$ are admissible ideals of $KQ'$ and $KQ$, respectively, and $\phi(I') \subset I$, we say that $\phi: (Q',I') \rightarrow (Q,I)$ is a \emph{morphism of quiver with relations}. In this case, $\phi$ induces a homomorphism of $K$-algebras $\phi: KQ'/I' \rightarrow KQ/I$.   


For a vertex $x \in Q$, consider the sets $$x^{+}=\{\alpha \in Q_1| s(\alpha)=x\} \ \  \mbox{and} \ \ x^-=\{\alpha \in Q_1| t(\alpha)=x\}.$$ 
One says that $Q$ is \emph{locally finite} if the sets $x^+$ and $x^-$ are both finite for all $x\in Q_0$. Note that if $Q$ is a locally finite quiver and $I$ is an admissible ideal of $KQ$, then $KQ/I$ is a locally bounded category. Moreover, if $Q$ is finite, then $KQ/I$ is a finite-dimensional algebra whose  the unit is $\sum_{x \in Q_0}e_x$, where $e_x=\varepsilon_x + I$ is the idempotent corresponding to the vertex $x$.

Given an algebra $A=KQ/I$ and let $x$ be a vertex of $Q$. We denote by $P_x=e_xA$ the corresponding indecomposable projective $A$-module associate to $x$. It is well-known that any arrow $\alpha: x\rightarrow y$ gives rise to a map $\alpha:P_y \rightarrow P_x$ given by left multiplication by $\alpha$. Moreover, this map gives a natural isomorphism $e_yAe_x \cong \mathrm{Hom}_A(P_x, P_y)$.

\subsection{Derived representation type}
It is well-known that for an algebra $A$ the category $D^b(A)$ can be identified with the homotopy category $\mathrm{K}^{-,b}(\mathrm{proj}\ A)$ of bounded above complexes of finitely generated projective right $A$-modules with bounded cohomologies. Recall that the \emph{cohomology dimension vector} of a complex $X \in D^b(A)$ is the vector $\mathrm{h-dim}(X)=(\mathrm{dim}_K H^i(X))_{i \in \mathbb{Z}}$, where $H^i(X)$ is the $i$-th cohomology space of $X$. 

\begin{definition}
Let $A$ be an algebra. 
	\begin{itemize}
		\item[1)] \cite{MR1913081} $A$ is said to be \emph{derived tame} if for any $n=(n_i)_{i \in \mathbb{N}}$ there exist a localization $R=K[x]_f$ with respect to some $f \in K[x]$ and a finite number of bounded complexes of $R$-$A$-bimodules $X_1, \cdots, X_k$ such that each $X^i_ j$ is finitely generated free over $R$ and (up to isomorphism) all but finitely many indecomposable object of cohomology dimension $n$ in $D^b(A)$ are of the form $S\otimes_R X_i$ for some $i=1, \cdots, n$ and some simple $R$-module $S$.  \\
		
		\item[2)] \cite{Dicotomy}  $A$ is said to be \emph{derived wild} if there exists a bounded complex $M$ of $K\langle x,y\rangle$-$A$-bimodules such that each $M^i$ is  free and of finite rank as left $K\langle x,y\rangle$-module and the functor $-\oplus_{K\langle x,y\rangle}M: \mathrm{mod}\, K\langle x,y\rangle \rightarrow \mathrm{mod}\ A$ preserves indecomposability and isomorphism classes. 
	\end{itemize} 
\end{definition} 

Two algebras $A$ and $B$ are said to be \emph{derived equivalent} if the respective derived categories $D^b(A)$ and $D^b(B)$ are equivalent as triangulated categories. Moreover, by an important result due to Rickard \cite{MR1002456},
this happens exactly when $B \cong (\mathrm{End} (T))^{op}$  where $T \in \mathrm{K}^b (\mathrm{proj}\ A)$ is a complex (called \emph{tilting complex}) satisfying:
	\begin{itemize}
		\item[(i)]for all $ i \not=0$, $\mathrm{Hom}_{\mathrm{K}^b(\mathrm{proj}\,A)}(T, T[i])=0$ (where $[-]$ denote the shift functor);
		\item[(ii)]$T$ generates $\mathrm{K}^b(\mathrm{proj}\ A)$ as a triangulated category.
	\end{itemize} 

\begin{remark} We recall from \cite[Theorem A]{MR1913081} that derived tameness is preserved under derived equivalence.
\end{remark}	

Let us recall from \cite{MR2497950} the following result. 

\begin{theorem}\label{rad}
Let $A$ be an algebra with radical squared zero, that is, $A\cong KQ/ \mathcal{R}_Q^2$. Then $A$ is derived tame if and only if $Q$ is either Dynkin or Euclidean graph.  \\
\end{theorem}

\subsection{Cleaving functors and Galois coverings} Let $F:B \rightarrow A$ be a $K$-linear functor between two locally bounded categories. The \emph{restriction functor} \linebreak $F_{\bullet}:\mathrm{Mod}\, A \rightarrow \mathrm{Mod}\, B$, $F_{\ast}(M)=MF$ admits a left adjoint $F^\bullet: \mathrm{Mod}\, B \rightarrow \mathrm{Mod}\, A$, called \emph{extension functor},  which is up to natural isomorphism well-defined by requiring that $F^\bullet$ is a right exact and coproduct preserving functor such that $F^\bullet A(-, a) = B(-, Fa)$ for all objects $a$ of $A$.  According to \cite{MR799266}, the functor $F$ is called \emph{cleaving} if the canonical natural transformation  $\Phi_F: \mathrm{Id}_{\mathrm{Mod}\, B} \rightarrow  F_\bullet F^\bullet$ admits a natural retraction, i.e., if there is a morphism $\Psi: F_\bullet F^\bullet \rightarrow \mathrm{Id}_{\mathrm{Mod}\, B}$  such that	$\Psi(M) \Phi_F(M) = \mathrm{Id}_M$ for each $M \in \mathrm{Mod}\,B$.\\

The best known examples of cleaving functors are the  Galois coverings. Following  \cite{MR853234}, a \emph{Galois covering} of a bounded quiver $(Q,I)$ is a morphism of quivers with relations  $\phi: (\widehat{Q}, \widehat{I}) \rightarrow (Q, I)$ together with a group $G$ of automorphisms of $(\widehat{Q}, \widehat{I})$ satisfying:
\begin{itemize}
	\item[i)] $G$ acts freely on $\widehat{Q}_0$;
	\item[ii)] $\phi g=\phi$ for every $g \in G$ and $\phi(x)=\phi(y)$ if, and only if,  $y=gx$ for some $g \in G$; 
	\item[iii)] $\phi$ induces  bijections $x^+ \longrightarrow p(x)^+$ e $x^- \longrightarrow p(x)^-$, for all $x \in \widehat{Q}_0$;
	\item[iv)] $I$ is the ideal generated by the elements of the form $p(\rho)$, with $\rho \in \widehat{I}$. \\
\end{itemize}

A Galois covering of bonded quiver $\phi:(\widehat{Q}, \widehat{I}) \rightarrow (Q, I)$ induces naturally a Galois covering functor between $K$-categories $F: K\widehat{Q}/\widehat{I} \rightarrow KQ/I$ in the sense of \cite{MR2170541}. \\ 

In this paper we have a particular interest on Galois coverings of \emph{monomial algebras}, that is, algebras of type $KQ/I$ where $I$ is generated by paths. Following \cite{MR654705} (see also \cite{MR724038}), given $(Q, I)$ a bounded quiver where $I$ is generated by paths, there is  Galois covering  $F:(\tilde{Q},\tilde{I})\rightarrow (Q,I)$ defined by fundamental group $\pi_1(Q, I)$. More precisely, fixed $x_0 \in Q$,  let $\pi_1(Q, I, x_0)$ be the set of reduce walks of $Q$ starting and ending at $x_0$. Of course the concatenation of walks induces a group structure to $\pi_1(Q, I, x_0)$. Moreover, since $Q$ is connected quiver, so $\pi_1(Q, I, x_0)\cong \pi_1(Q, I, x_1)$ for any $x_1 \in Q_0$. Thus $\pi_1(Q, I, x_0)$ will be denoted by $\pi_1(Q, I)$, and we say that it is the \emph{fundamental group} of $(Q, I)$. 

Now, the vertices of $\tilde{Q}$ are the reduced walks starting at a fixed vertex $x_0\in Q_0$, and for all $w, w \in \tilde{Q}_0$  there is an arrow $[\alpha, w] \in \tilde{Q}_1$ from $w$ to $w'$ whenever $ w'= w\alpha$ with $\alpha \in Q_1$. Note that there is at most one arrow $[\alpha, w]:w\rightarrow w'$ since the arrow $\alpha$ is unique, if it exists. This provides a quiver-morphism $p: \tilde{Q} \rightarrow Q$ defined by $p(w):=t(w)$ and $p([\alpha, w])=\alpha$. Finally, the ideal $I$ is defined to be generated by the inverse images under $p$ of the generators of $I$.

\begin{remark}
\begin{itemize}
	\item[(1)] Note that if $(Q, I)$ is bounded quiver where $Q$ is connected and $I$ is generated by paths, then $\tilde{Q}$ is a connected tree and the ideal $\tilde{I}$ is generated by paths. Moreover, if $Q$ has cycles, then $\tilde{Q}$ is a locally finite quiver with an infinite number of vertices.\\
	\item[(2)] In this paper, all Galois coverings are supposed to be Galois covering defined by fundamental group.\\
	\item[(3)] It is well-known that Galois coverings are cleaving functors (see \cite[3.2]{MR332887} for example).
\end{itemize}	
  \end{remark} 

The next result follows from \cite[3.8(b)]{MR799266}.
\begin{lemma}
	\begin{itemize}
		\item[(a)] Let $\mathcal{C}$ be a locally bounded category and  $\mathcal{C'}$ a full subcategory of $\mathcal{C}$. The inclusion  $\mathcal{C'} \hookrightarrow \mathcal{C}$ is a cleaving functor.
		\item[(b)] The composition of two cleaving functors is cleaving.
	\end{itemize}
\end{lemma}

Let us recall from \cite{MR3778425} the following theorem. 
\begin{theorem}\label{Zhang}
Let $F: B \rightarrow A$ be a cleaving functor between bounded categories with $\mathrm{gl.dim}\,B< \infty$. Then $B$ derived wild  implies $A$ derived wild.  
\end{theorem}

\begin{lemma}\label{lemma1}
Let $A=KQ/I$ be  a  monomial algebra and $\tilde{A}=K\tilde{Q}/\tilde{I}$ its Galois covering.

\begin{itemize}
	\item[(a)] If $\tilde{A}$ has a full, bounded and derived wild subcategory $\widehat{A}$, then $A$ is also derived wild.
	\item[(b)] If $\tilde{A}$ has a full and bounded subcategory $\widehat{A}\cong K \Delta/ \mathcal{R}_{\Delta}^2$ such that $\Delta$ is neither a Dynkin nor an Euclidean graph, then $A$ is derived wild.
\end{itemize} 
\end{lemma}
\begin{proof}
(a) Let $\widehat{F}: \widehat{A} \rightarrow A$ be the composition of the canonical inclusion  $\widehat{A} \hookrightarrow \tilde{A}$ and the Galois covering $F:\tilde{A} \rightarrow A$, thus $\widehat{F}$ is a cleaving functor. Since $\tilde{Q}$ is a tree and $\widehat{A}$ is bounded, then $\mathrm{gl.dim}\, \widehat{A}< \infty$. Therefore $A$ is derived wild  by Theorem \ref{Zhang}.\\

(b) We get that $\widehat{A}$ is derived wild by Theorem \ref{rad}, hence $A$ is derived wild by item (a).
\end{proof}

\begin{example}Let $A=KQ/I$ be the algebra where
	
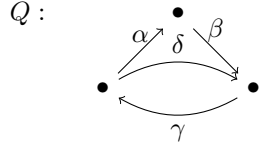
\begin{center}		
		\begin{tikzpicture}
		\node  (d) at (-1, 1) {$Q:$};
		\node  (a) at (0, 0) {$\bullet$};
		\node  (b) at (2, 0) {$\bullet$};
		\node  (c) at (1, 1) {$\bullet$};
		\draw [->] (a) -- node[above] { $\alpha$}  (c) ;
		\draw [->] (c) -- node[above] { $\beta$} (b) ;
		\draw [->] (a) to [out=30,in=150] node[above] { $\delta$} (b);
		\draw [->] (b) to [out=210,in=330] node[below] { $\gamma$} (a);
		\end{tikzpicture}	
	\end{center}
and $I=\langle\delta\gamma, \alpha\beta, \beta\gamma\rangle$. Putting $h=\gamma\alpha$, the Galois covering of $A$ has a full bounded subcategory $\widehat{A}=K\widehat{Q}/\mathcal{R}_Q^2$ where
	
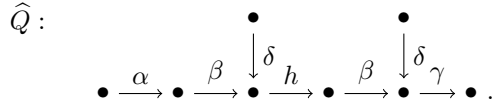
\begin{center}
		\begin{tikzpicture}
		\node  (j) at (-1, 1) {$\widehat{Q}:$};
		\node  (a) at (0, 0) {$\bullet$};
		\node  (b) at (1, 0) {$\bullet$};
		\node  (c) at (2, 0) {$\bullet$};
		\node  (d) at (3, 0) {$\bullet$};
		\node  (e) at (4, 0) {$\bullet$};
		\node  (f) at (5, 0) {$\bullet$ .};
		\node  (g) at (2, 1) {$\bullet$};
		\node  (h) at (4, 1) {$\bullet$};
		\draw [->] (a) -- node[above] { $\alpha$}(b) ;
		\draw [->] (b) -- node[above] { $\beta$} (c) ;
		\draw [->] (c) -- node[above] { $h$}(d) ;
		\draw [->] (d) -- node[above] { $\beta$}  (e) ;
		\draw [->] (e) -- node[above] { $\gamma$}  (f);
		\draw [->] (g) -- node[right] { $\delta$}  (c) ;
		\draw [->] (h) -- node[right] { $\delta$}  (e) ;
		\end{tikzpicture}	
	\end{center}
Since $\tilde{A}$ is a tree, then $\mathrm{gl.dim}\,\tilde{A}<\infty$. Moreover, it is derived wild by Theorem \ref{rad}. Hence $A$ is derived wild by Lemma \ref{lemma1}.\\
\end{example}

\subsection{Mutations} We recall the notion of mutations of algebras from \cite{Perverce}. These are local operations on an algebra $A$ producing new algebras derived equivalent to $A$. 

Let $A= KQ/I$ be an algebra and let $x \in Q_0$ be a vertex without loops. Consider the following complexes: 
\[\xymatrix{
	R_x: & \cdots \ \ 0 \ar[r] &  P_x\ar[r]^-{f} & \bigoplus_{j \rightarrow x} P_j \ar[r]& 0 \ar[r] & 0 \ \ \cdots ,\\
	L_x: &\cdots \ \ 0 \ar[r] & 0 \ar[r] & \bigoplus_{x \rightarrow j} P_j\ar[r]^-{g} & P_x \ar[r]&0 \ \ \cdots ,\\
	R_i:&\cdots \ \ 0 \ar[r]  &0\ar[r] & P_i\ar[r] &0 \ar[r]& 0 \ \ \cdots,
}
\]
where the map $f$ is induced by all maps $P_x \rightarrow P_j$ corresponding to the arrows $j\rightarrow x$, the map $g$ is induced by the maps $P_j \rightarrow P_x$ corresponding to the arrows $x \rightarrow j$, the term $P_x$ lies in degree $-1$ in $R_x$ and in degree $1$ in $L_x$, and all other terms are in degree $0$.

\begin{definition}\label{defi02}
Let $A=KQ/I$ be an algebra and let $x \in Q_0$ be a vertex without loops.
\begin{itemize}
	\item[(a)] We say that the negative mutation of $A$ at $x$ is defined if \linebreak $T_x^-(A)=\bigoplus_{i \in Q_0} R_i \in \mathrm{K}^b(\mathrm{proj}\,A)$ is a tilting complex. In this case, we call the algebra $\mu_x^-(A)=(\mathrm{End}(T_x^-(A)))^{op}$ the \emph{negative mutation} of $A$ at the vertex $x$.
	
	\item[(b)] We say that the positive mutation of $A$ at $x$ is defined if\linebreak  $T_x^+(A)=\left(\bigoplus_{i \not= x} R_i\right)\oplus L_x  \in \mathrm{K}^b(\mathrm{proj}\,A)$ is a tilting complex. In this case, we call the algebra $\mu_x^+(A)=(\mathrm{End}(T_x^+(A)))^{op}$ the \emph{positive mutation} of $A$ at the vertex $x$.
\end{itemize}
\end{definition}

By Rickard's Theorem, the negative and positive mutations of an algebra $A$ at a vertex, when defined, are always derived equivalent to $A$.

\begin{remark} 
If $x$ is a sink (that is, $x^+=\emptyset)$, follows by \cite[Proposition 2.3]{Perverce} that $\mu_x^-(A)$ is defined. Moreover, in this case $T_x^-(A)$ is isomorphic in $\mathrm{D^b}(A)$ to the APR-tilting module corresponding to $x$ (see \cite{MR530043}). Similarly, if $x$ is a source (that is, $x^-=\emptyset$) then $\mu_x^+(A)$ is defined.    
\end{remark}
  
\subsection{Blowing-up} Let $A= KQ/I$ be an algebra and let  $D$ be a set of vertices without loops of $Q$. Following \cite{MR1868357}, the \emph{blowing-up} of $A$ at $D$ is the algebra $A[D]= KQ[D]/I[D]$, where $Q[D]$ and $I[D]$ are describe below. 

The quiver $Q[D]$ is obtained from $Q$ replacing each  vertex $d\in D$ by two vertices $d^-$ and $ d^+$,  each arrow $\alpha:x\rightarrow d$ with $d \in D$ by two arrows $\alpha^-:x \rightarrow d^-$ and $\alpha^+: x \rightarrow d^+$ and dually for each arrow $\beta: d \rightarrow x$.
 
There is an obvious quiver epimorphism $p:Q[D]\rightarrow Q$ witch extends uniquely to an epimorphism of algebras $p: KQ[D]\rightarrow KQ$. The ideal $I[D]$ of $KQ[D]$ is defined as the ideal generated by the inverse image of the generators of $I$ under $p$. 

\begin{example} Let $A= KQ/I$ be given by the bounded quiver: 
	
	\begin{center}
		\begin{tikzpicture}
		\node  (j) at (-1, 0) {$Q:$};
		\node  (1) at (0, 0) {$1$};
		\node  (2) at (1.5, 0) {$2$};
		\node  (3) at (3, 0) {$3$};
		\node  (4) at (4.5, 0) {$4$};
		\node  (5) at (6,0)  {$5$};
		\node  (i) at (8, 0) {and \, $I=\langle \alpha \beta \rangle$.};
		\draw [->] (1) -- node[above] { $\alpha$}(2) ;
		\draw [->] (2) -- node[above] { $\beta$} (3) ;
		\draw [->] (3) -- node[above] { $\gamma$}(4) ;
		\draw [->] (4) -- node[above] { $\delta$}(5) ;
		\end{tikzpicture}	
	\end{center}

If $D=\{1,3\}$, then  the blowing-up $A[D]$ is gives by the following quiver:
	
	\begin{center}
		\begin{tikzpicture}
		\node  (j) at (-1, 0) {$Q[D]:$};
		\node  (a) at (0, 0) {$1^+$};
		\node  (b) at (0, -2) {$1^-$};
		\node  (c) at (2, -1) {$2$};
		\node  (d) at (4, 0) {$3^+$};
		\node  (e) at (4, -2) {$3^-$};	
		\node  (f) at (6, -1) {$4$};
		\node  (g) at (8, -1 ) {$5$};
		\draw [->] (a) -- node[above] { $\alpha^{+}$}(c) ;
		\draw [->] (b) -- node[below] { $\alpha^{-}$} (c) ;
		\draw [->] (c) -- node[above] { $\beta^{+}$}(d) ;
		\draw [->] (c) -- node[below] { $\beta^{-}$}(e) ;
		\draw [->] (d) -- node[above] { $\gamma^{+}$}(f) ;
		\draw [->] (e) -- node[below] { $\gamma^{-}$} (f) ;
		\draw [->] (f) -- node[above] { $\delta$}(g) ;
		\end{tikzpicture}	
	\end{center}
bounded by $I[D]=\langle \alpha^+\beta^+, \alpha^-\beta^-, \alpha^+\beta^-, \alpha^-\beta^+, \beta^+\gamma^+ - \beta^-\gamma^-  \rangle $.
\end{example}

\begin{remark}\label{remark-1} Given $A=K Q/I$ and  $D \subset Q_0$ a set of vertices without loops, let $p:Q[D]\longrightarrow Q$ be the canonical quiver epimorphism. Since $p(I[D]) \subset I$, then  $p$ induces a $K$-linear functor $\pi: A[D] \rightarrow A$ where $\pi(i)= p_0(i)$ for each $i \in Q[D]_0$ and  $\pi_{i,j}(u)=\pi(u)$ for all $u \in A[D](i,j)$. Moreover, by construction of $A[D]$ we have that $\pi_{i,j}$ is a monomorphism for all $i,j \in Q[D]_0$ and 
\[\left\{\begin{array}{ll}
\mathrm{Im}\,\pi_{i, j}= A(p_0(i),p_0(j)), &  \mbox{ if } (i,
j) \notin \mathfrak{D};\\
\mathrm{Im}\, \pi_{i,j}=\mathrm{rad}\, A (p_0(i),p_0(j)), & \mbox{ if } (i,j) \in \mathfrak{D};
\end{array}\right.
\]	
where $\mathfrak{D}=\{(d^+,d^-), (d^-, d^+)| d \in D\}$.  
\end{remark}

\begin{lemma}\label{lemma12}
Under the considerations of Remark \ref{remark-1}: 
\begin{itemize}
	\item[(a)] The functor $\pi: A[D]\rightarrow A$ induces an exact functor $$\mathrm{K}(\pi): \mathrm{K}^b(\mathrm{proj}\,A[D]) \rightarrow \mathrm{K}^b(\mathrm{proj}\,A).$$
	
	\item[(b)] Let $B$ be a locally bonded $K$-category such that  there is a $K$-linear functor $\varphi:B \rightarrow A$  and a bijection $l: B_0 \rightarrow Q[D]_0$  such that:
	
	\begin{itemize}
		\item $\varphi(a)=p_0(l(a))$, for all $a \in B_0$;
		\item for all $a, b \in B_0$, the map $\varphi_{a,b}:B(a, b) \rightarrow A(\varphi(a), \varphi(b)) $ is a monomorphism and 
		\[\left\{\begin{array}{ll}
		\mathrm{Im}\,\varphi_{a, b}= A(p_0(l(a)),p_0(l(b))), &  \mbox{ if } (l(a),
		l(b)) \notin \mathfrak{D};\\
		\mathrm{Im}\, \varphi_{a,b}=\mathrm{rad}\, A (p_0(l(a)),p_0(l(b))), & \mbox{ if } (l(a),l(b)) \in \mathfrak{D}.
		\end{array}\right.
		\]
	\end{itemize}     
\end{itemize}
Then $B\cong A[D]$.
\end{lemma}

\begin{proof} (a)  For all $i \in Q[D]_0$ we denote by $\widehat{P}_i= e_iA[D]$ the corresponding indecomposable projective $A[D]$-module. Thus $\pi$ induces, up to isomorphism, a coproduct preserving additive functor $\Pi:\mathrm{proj}\, A[D] \rightarrow \mathrm{proj}\, A$ such that $\Pi(\widehat{P}_i)= P_{p_0(i)}$ for each $i \in Q[D]_0$, and
	\[\begin{array}{rrl}
	\Pi_{\widehat{P}_i,\widehat{P}_j}:& \mathrm{Hom}_{A[D]}(\widehat{P}_i, \widehat{P}_j) & \longrightarrow \mathrm{Hom}_{A}(P_{p_0(i)}, P_{p_0(j)})\\
	& x  &\mapsto \pi(x)
	\end{array}
	\]
for all $i,j \in Q[D]_0$.
	
It follows from \cite[Proposition 1.1.1, p.192]{Milicic} that $\Pi$ induces an exact functor $\mathrm{K}(\Pi): \mathrm{K}^b(\mathrm{proj}\,A[D]) \longrightarrow \mathrm{K}^b(\mathrm{proj}\,A)$ defined as belongs.  

Given a complex $X \in \mathrm{K}^b(\mathrm{proj}\,A[D])$ given by 
$$X: \xymatrix{ \cdots  X^{n-1} \ar[rr]^{d_X^{n-1}} && X^{n} \ar[rr]^{d_X^{n}} && X^{n}  \cdots},$$
we get that
$$\mathrm{K}(\Pi)(X): \xymatrix{ \cdots  \Pi(X^{n-1}) \ar[rr]^{\Pi(d_X^{n-1})} && \Pi(X^{n}) \ar[rr]^{\Pi(d_X^{n+1})} && \Pi(X^{n})  \cdots \, .}$$

If $f \in \mathrm{K^b}(\mathrm{proj}\,A[D])(X, Y)$ is given by 
$$\xymatrix{X:\ar[d]^f&  \cdots  X^{n-1}\ar[d]^{f^{n-1}} \ar[rr]^{d_X^{n-1}} && X^{n}\ar[d]^{f^n} \ar[rr]^{d_X^{n}} && X^{n+1}\ar[d]^{f^{n+1}}  \cdots\\
	Y: &\cdots  Y^{n-1} \ar[rr]^{d_Y^{n-1}} && Y^{n} \ar[rr]^{d_Y^{n}} && Y^{n+1}  \cdots,}$$
then the morphism $\mathrm{K}(\Pi)(f) \in \mathrm{K^b}(\mathrm{proj}\,A)(\mathrm{K}(\Pi)(X), \mathrm{K}(\Pi)(Y))$ is
$$\xymatrix{\mathrm{K}(\Pi)(X):\ar[d]^{\mathrm{K}(\Pi)(f)}&  \cdots  \Pi(X^{n-1})\ar[d]^{\Pi(f^{n-1})} \ar[rr]^{\Pi(d_X^{n-1})} && \Pi(X^{n})\ar[d]^{\Pi(f^n)} \ar[rr]^{\Pi(d_X^{n})} && \Pi(X^{n+1})\ar[d]^{\Pi(f^{n+1})}  \cdots\\
	\mathrm{K}(\Pi)(Y): &\cdots  \Pi(Y^{n-1}) \ar[rr]^{\Pi(d_Y^{n-1})} && \Pi(Y^{n}) \ar[rr]^{\Pi(d_Y^{n})} && \Pi(Y^{n+1})  \cdots .}$$\\

(b) We define a functor $F:B \rightarrow A[D]$ as follows. For all $a \in B_0$, put $F(a)=l(a)$. Since $\pi_{l(a), l(b)}$ is a monomorphism and $\mathrm{Im}\,\varphi_{a,b}=\mathrm{Im}\,\pi_{l(a),l(b)}$ for every $a,b \in B_0$, then we can define: 
\[\begin{array}{rrl}
F_{a,b}:& B(a,b) & \longrightarrow A[D](l(a),l(b))\\
& x  &\mapsto \pi_{l(a),l(b)}^{-1}(\varphi_{a, b}(x)).
\end{array}
\]
Note that $F_{a,b}$ is $K$-linear since $\varphi_{a, b}$ and $\pi_{l(a),l(b)}$ are $K$-linear. Moreover, for all $a,b,c \in B_0$, $x \in B(a,b)$ and $y \in B(b,c)$ we have that:
\begin{eqnarray*}
	\pi_{l(a),l(c)} (F_{b,c}(y) F_{a,b}(x)) & = & \pi_{l(a),l(c)}\left( (\pi_{l(b),l(c)}^{-1}\varphi_{b,c}(y))(\pi_{l(a),l(b)}^{-1}\varphi_{a, b}(x))\right)\\
	& = &  \left(\pi_{l(b),l(c)} \pi_{l(b),l(c)}^{-1}\varphi_{b,c}(y)\right) \left(\pi_{l(a),l(b)} \pi_{l(a),l(b)}^{-1}\varphi_{a,b}(x)\right)\\
	& = &  \varphi_{b,c}(y)\varphi_{a,b}(x)\\
	& = & \varphi_{a,c}(yx)\\
	& = & \pi_{l(a),l(c)}\pi_{l(a),l(c)}^{-1}\varphi_{a,c}(yx)\\
	& = & \pi_{l(a),l(c)}(F_{a,c}(yx)).
\end{eqnarray*}  
Since $\pi_{l(a),l(c)}$ is a monomorphism then $F_{a,c}(yx)= F_{b,c}(y) F_{a,b}(x)$. Therefore $F$ is a  $K$-linear functor.

Finally, it is clear that $F$ is an isomorphism since $F$ is bijective on the objects and morphisms.

\end{proof} 

\begin{lemma}\label{lemma3}
Let $A=KQ/I$ be an algebra and let $D\subset Q_0$ be a set of vertices without loops.
\begin{itemize}
	\item[(a)]If $x \in Q_0$ is a sink, $x \notin D$ and  for any $d \in D$ there is no arrow  $d \rightarrow x $, then $\mu_x^-(A[D])\cong \mu_x^-(A)[D]$.  In particular, $A[D]$ is derived equivalent to $ \mu_x^-(A)[D]$.
	\item[(b)]If $x \in Q_0$ is a source,  $x \notin D$ and for any $d \in D$ there is no arrow $ x \rightarrow d$, then  $\mu_x^+(A[D])\cong \mu_x^+(A)[D]$ . In particular, $A[D]$ is derived equivalent to $ \mu_x^+(A)[D]$.
\end{itemize} 
\end{lemma}

\begin{proof} We only prove (a), the proof of (b) is similar. For to prove the item (a) we use induction on the number of elements of $D$.

Suppose first that $D=\{d\}$. Let $T_x^-(A)= \bigoplus_{i \in Q_0} R_i\in \mathrm{K}^b(\mathrm{proj}\, A)$ and $\mu_x^-(A)= \left(\mathrm{End}(T_x^-(A))\right)^{op}$ be as in Definition \ref{defi02}. 
Since $R_i$ is indecomposable for any $i \in Q_0$, then $\mu_{x}^-(A)$ is the $K$-category whose set of objects  is $Q_0$ and for all $i,j \in Q_0$ we have $\mu_x^-(A)(i,j)= \mathrm{Hom}_{\mathrm{K}^b(\mathrm{proj}\,A)}(R_j, R_i)$. Moreover, suppose that  $\mu_{x}^-(A) \cong K \Delta/ J$ and let $p:(Q[d], I[d]) \longrightarrow (Q, I)$ and $q:(\Delta[d], J[d]) \longrightarrow (\Delta, J)$ be the canonical quiver-morphisms. Hence $Q[d]_0= \Delta[d]_0$ and the maps $p_0=Q[d]_0 \rightarrow Q_0$ and $q_0=\Delta[d]_0 \rightarrow \Delta_0$ are equal.
	
In the other hand, given $\widehat{P}_i$ the projective $A[d]$-module associated to  $i \in Q[d]_0$, we define the complexes of $A[d]$-modules:
	$$\xymatrix{
		\widehat{R}_i: & \cdots \ \ 0 \ar[r] & 0 \ar[r]& \widehat{P}_i \ar[r]& 0 \ \ \cdots, & i \not=x,\\ 
		\widehat{R}_x: & \cdots \ \ 0 \ar[r] &  \widehat{P}_x\ar[r]^-{\widehat{f}} & \displaystyle\bigoplus_{j \rightarrow x} \widehat{P}_j \ar[r]& 0 \ \ \cdots ,\\
	}
	$$ 
concentrated in degree $-1$ and $0$, where the map $\widehat{f}$ is induced by all the arrows \linebreak$\alpha \in Q[d]_1$ such that $t(\alpha)=x$. By definition $T_x^-(A[d])= \bigoplus_{i \in Q[d]_0} \widehat{R}_i$ and $\mu_x^-(A[d])=\left(\mathrm{End}(T_x^-(A[d]))\right)^{op}$, thus $\mu_x^-(A[d])$ is the $K$-category whose set of objects is  $Q[d]_0$ and for every $i,j \in Q[D]_0$ we have $\mu_x^-(A[d])(i,j)= \mathrm{Hom}_{\mathrm{K}^b(\mathrm{proj}\,{A[d]})}(\widehat{R}_j, \widehat{R}_i)$. Hence there is a bijection, which we suppose to be the identity, between the objects  of  $\mu_{x}^-(A[d])$ and the objects of $\mu_{x}^-(A)[d]$.  
	
Given  the morphism  $p:(Q[d],I[d]) \rightarrow (Q,I)$, let  $\pi: A[d] \rightarrow A$ be the homomorphism of algebras induced by $p$. From the Lemma \ref{lemma12}(a), $\pi$ induces an additive functor $\mathrm{K}(\Pi):\mathrm{K}^b(\mathrm{proj}\,A[d]) \longrightarrow \mathrm{K}^b(\mathrm{proj}\,A)$. Since there is no arrow $d \rightarrow x$, then $P_d$ is not a  summand of $R_x^0$ and $\widehat{P}_{d^+},\widehat{P}_{d^-} $ are not summands of $\widehat{R}^0_x$, thus $\mathrm{K}(\Pi)(\widehat{R}_i)=R_{p_0(i)}$ for all $i \in Q[d]_0$. Hence we can define a functor $\tilde{\pi}: \mu_{x}^-(A[d]) \rightarrow \mu_{x}^-(A)$ where  $\tilde{\pi}(i)= p_0(i)$ for every $i \in Q[d]_0$, and for $g \in \mu_{x}^-(A[d])(i,j)$ we have $\tilde{\pi}_{i,j}(g):= \mathrm{K}(\Pi)(g) \in  \mu_{x}^-(A)(p_0(i),p_0(j))$.   
	
For $\varepsilon \in \{+, -\}$, denote by $A^{\varepsilon}$ the full subcategory of $A[d]$ defined by $Q[d]_0\backslash \{d^{-\varepsilon}\}$ and denote by $B^{\varepsilon}$ the full subcategory of $\mu_{x}^-(A[d])$ defined by $Q[d]_0 \backslash \{d^{-\varepsilon}\}$. Of course $\pi|_{A^{\varepsilon}}: A^{\varepsilon} \rightarrow A$ is an isomorphism, thus $\tilde{\pi}|_{B^{\varepsilon}}:B^{\varepsilon} \rightarrow \mu_{x}^-(A)$ is an isomorphism too. 
	
Given $i,j \in  Q[D]_0$ such that $\{i,j\} \not= \{d^+,d^-\}$, then $i,j \in Q[D]_0\backslash\{d^{-\varepsilon}\}$ for any $\varepsilon \in \{+, -\}$, hence $\tilde{\pi}_{i,j}= \tilde{\pi}|_{B^{\varepsilon}}$ is an isomorphism. Moreover, since \linebreak$\widehat{R}_{d^\varepsilon}= \xymatrix{0 \ar[r] & \widehat{P}_{d^\varepsilon}\ar[r]& 0}$, so  $\tilde{\pi}_{d^{\varepsilon}, d^{-\varepsilon}}= \pi_{d^{\varepsilon}, d^{-\varepsilon}}$  and therefore it is a monomosphism and $\mathrm{Im}\, \tilde{\pi}_{d^{\varepsilon}, d^{-\varepsilon}} = \mathrm{Im}\, \pi_{d^{\varepsilon}, d^{-\varepsilon}} = \mathrm{rad}\, A(d,d) = \mathrm{rad}\, \mu_{x}^-(A)(d,d)$. Thus \linebreak$\mu_{x}^-(A[d]) \cong \mu_{x}^-(A)[d]$ by Lemma \ref{lemma12}(b).
	
Now, suppose that $|D|=n>1$.  Fixed $d_0 \in D$, put $D_0= D\backslash \{d_0\}$. In this way, $\mu_x^-(A[D])= \mu_{x}^-((A[d_0])[D_0]) \cong \mu_{x}^-((A[d_0]))[D_0]$ by induction hypotheses. Moreover, since   $\mu_{x}^-(A[d_0]) \cong \mu_{x}^-(A)[d_0]$, then $\mu_x^-(A[D]) \cong \mu_{x}^-((A[d_0]))[D_0] \cong (\mu_{x}^-(A)[d_0])[D_0] = \mu_{x}^-(A)[D]$.
\end{proof}

\section{Good Quadratic String Algebras}

\begin{definition}
Let $A=KQ/I$ be a quadratic string algebra. A vertex $x \in Q_0$ is called \emph{gentle} if the following conditions holds:
	\begin{itemize}
		\item[(i)]for each arrow $\alpha$ in $Q$ such that $s(\alpha)=x$, there is at most one arrow $\beta$ such that $t(\beta)=x$ and $\alpha \beta \in I$;
		\item[(ii)]for each arrow $\alpha$ in $Q$ such that $t(\alpha)=x$, there is at most one arrow $\gamma$ such that $s(\gamma)=x$ and $\gamma\alpha \in I$.
	\end{itemize}
\end{definition}

Following \cite{MR892057}, $A$ is said to be \emph{gentle} if every vertex of $Q$ is gentle. 

In the next definition we will use the following notation. Given a quiver $Q$, then in its graphic presentation every sink will be denoted by $\circ$ and every source will be denoted by $\diamond$. The other vertices will be denoted by $\bullet$.    

\begin{definition}\label{defi01}
Let $(Q, I)$ be a quadratic string pair. For every $i=1, \cdots, 6$ we define the sets $E_i, O_i \subset Q_0$ as follows. \\

(1) We say that a vertex $x$ belongs to  $E_1$ if there are arrows $\alpha=\alpha_x, \beta=\beta_x, \gamma=\gamma_x, \delta=\delta_x \in Q_1$ such that:
\begin{itemize}
	\item $t(\alpha)=t(\beta)=s(\gamma)=s(\delta)=x$;
	\item $s(\alpha), s(\beta), x, t(\gamma)$ and $t(\delta)$ are pairwise distinct;
	\item $s(\alpha)^+ = \{\alpha\}$ and $s(\alpha)^-=\emptyset$;
	\item $t(\gamma)^+ = \emptyset$ and $t(\gamma)^-=\{\gamma\}$;
	\item $\alpha\gamma, \beta\gamma, \alpha\delta \in I$ but $\beta\delta \notin I$.
\end{itemize}
Moreover, set $O_1=\{s(\alpha_x), t(\gamma_x) | x \in E_1 \}$. Therefore, if $x \in E_1$ then its neighborhood  has the following shape:
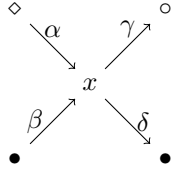
\begin{center}
	\begin{tikzpicture}
	\node  (a) at (0, 1) {$\diamond$};
	\node  (b) at (1, 0) {$x$};
	\node  (c) at (0, -1) {$\bullet$};
	\node  (d) at (2, 1) {$\circ$};
	\node  (e) at (2, -1) {$\bullet$};
	\draw [->] (a) -- node[above] { $\alpha$} (b);
	\draw [->] (c) -- node[left] { $\beta$} (b);
	\draw [->] (b) -- node[above] { $\gamma$} (d);
	\draw [->] (b) -- node[right] { $\delta$} (e);
	\end{tikzpicture}
\end{center}

(2) We say that a vertex $x$ belongs to  $E_2$ if there are arrows $\alpha=\alpha_x, \beta=\beta_x, \gamma=\gamma_x, \delta=\delta_x \in Q_1$ such that:
\begin{itemize}
	\item $t(\alpha)=t(\beta)=s(\gamma)=s(\delta)=x$;
	\item $s(\beta)=t(\gamma)$;
	\item $s(\alpha), s(\beta), x, t(\gamma)$ are pairwise distinct;
	\item $s(\alpha)^+ = \{\alpha\}$ and $s(\alpha)^-=\emptyset$;
	\item $t(\gamma)^+ = \emptyset$ and $t(\gamma)^-=\{\gamma\}$;
	\item $\alpha\gamma, \beta\gamma, \alpha\delta, \delta\beta \in I$ but $\beta\delta \notin I$.
\end{itemize}
Moreover, set $O_2=\{s(\alpha_x), t(\gamma_x) | x \in E_2 \}$.  Therefore, if $x \in E_2$ then its neighborhood has the following shape:
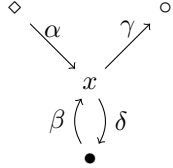
\begin{center}
	\begin{tikzpicture}
	\node  (a) at (0, 1) {$\diamond$};
	\node  (b) at (1, 0) {$x$};
	\node  (c) at (1, -1) {$\bullet$};
	\node  (d) at (2, 1) {$\circ$};
	\draw [->] (a) -- node[above] { $\alpha$} (b);
	\draw [->] (c) to [bend left] node[left] { $\beta$} (b);
	\draw [->] (b) to [bend left] node[right] { $\delta$} (c);
	\draw [->] (b) -- node[above] { $\gamma$} (d);
	\end{tikzpicture}
\end{center}

(3) We say that a vertex $x$ belongs to  $E_3$ if there are arrows $\alpha=\alpha_x, \beta=\beta_x, \gamma=\gamma_x \in Q_1$ such that:
\begin{itemize}
	\item $t(\alpha)=t(\beta)=s(\gamma)=x$;
	\item $s(\alpha), s(\beta), x, t(\gamma)$ are pairwise distinct;
	\item $s(\alpha)^+ = \{\alpha\}$ and $s(\alpha)^-=\emptyset$;
	\item $s(\beta)^+ = \{\beta\}$ and $s(\beta)^-=\emptyset$;
	\item $x^+=\{\gamma\}$;
	\item $\alpha\gamma, \beta\gamma \in I$.
\end{itemize}
Moreover, set $O_3=\{s(\alpha_x), s(\beta_x) | x \in E_3 \}$.  Therefore, if $x \in E_3$ then its neighborhood  has the following shape:
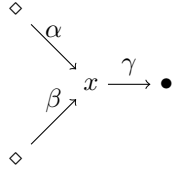
\begin{center}
	\begin{tikzpicture}
	\node  (a) at (0, 1) {$\diamond$};
	\node  (b) at (0, -1) {$\diamond$};
	\node  (c) at (1, 0) {$x$};
	\node  (d) at (2, 0) {$\bullet$};
	\draw [->] (a) -- node[above] { $\alpha$} (c);
	\draw [->] (b) -- node[above] { $\beta$} (c);
	\draw [->] (c) -- node[above] { $\gamma$} (d);
	\end{tikzpicture}
\end{center}

(4) We say that a vertex $x$ belongs to  $E_4$ if there are arrows $\alpha=\alpha_x, \beta=\beta_x, \gamma=\gamma_x \in Q_1$ such that:
\begin{itemize}
	\item $t(\alpha)=t(\beta)=s(\gamma)=x$;
	\item $s(\alpha), s(\beta), x, t(\gamma)$ are pairwise distinct;
	\item $s(\alpha)^+ = \{\alpha\}$ and $s(\alpha)^-=\emptyset$;
	\item $t(\gamma)^+ = \emptyset$ and $t(\gamma)^-=\{\gamma\}$;
	\item $x^+=\{\gamma\}$;
	\item $\alpha\gamma, \beta\gamma \in I$.
\end{itemize}
Moreover, set $O_4=\{s(\alpha_x), t(\gamma_x) | x \in E_4 \}$. Therefore, if $x \in E_4$ then  its neighborhood  has the following shape:
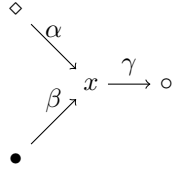
\begin{center}
	\begin{tikzpicture}
	\node  (a) at (0, 1) {$\diamond$};
	\node  (b) at (0, -1) {$\bullet$};
	\node  (c) at (1, 0) {$x$};
	\node  (d) at (2, 0) {$\circ$};
	\draw [->] (a) -- node[above] { $\alpha$} (c);
	\draw [->] (b) -- node[above] { $\beta$} (c);
	\draw [->] (c) -- node[above] { $\gamma$} (d);
	\end{tikzpicture}
\end{center}

(5) We say that a vertex $x$ belongs to  $E_5$ if there are arrows $\alpha=\alpha_x, \gamma=\gamma_x, \delta=\delta_x \in Q_1$ such that:
\begin{itemize}
	\item $t(\alpha)=s(\gamma)=s(\delta)=x$;
	\item $s(\alpha), x, t(\gamma), t(\delta)$ are pairwise distinct;
	\item $t(\delta)^+ = \emptyset$ and $t(\delta)^-=\{\delta\}$;
	\item $t(\gamma)^+ = \emptyset$ and $t(\gamma)^-=\{\gamma\}$;
	\item $x^-=\{\alpha\}$;
	\item $\alpha\gamma, \alpha\delta \in I$.
\end{itemize}
Moreover, set $O_5=\{t(\delta_x), t(\gamma_x) | x \in E_5 \}$. Therefore, if $x \in E_5$ then its the neighborhood  has the following shape:
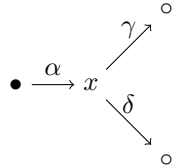
\begin{center}
	\begin{tikzpicture}
	\node  (a) at (0,0) {$\bullet$};
	\node  (b) at (2, -1) {$\circ$};
	\node  (c) at (1, 0) {$x$};
	\node  (d) at (2, 1) {$\circ$};
	\draw [->] (a) -- node[above] { $\alpha$} (c);
	\draw [->] (c) -- node[above] { $\delta$} (b);
	\draw [->] (c) -- node[above] { $\gamma$} (d);
	\end{tikzpicture}
\end{center}

(6) We say that a vertex $x$ belongs to  $E_6$ if there are arrows $\alpha=\alpha_x, \gamma=\gamma_x, \delta=\delta_x \in Q_1$ such that:
\begin{itemize}
	\item $t(\alpha)=s(\gamma)=s(\delta)=x$;
	\item $s(\alpha), x, t(\gamma), t(\delta)$ are pairwise distinct;
	\item $s(\alpha)^+=\{\alpha\}$ and $s(\alpha)^-=\emptyset$;
	\item $t(\gamma)^+ = \emptyset$ and $t(\gamma)^-=\{\gamma\}$;
	\item $x^-=\{\alpha\}$;
	\item $\alpha\gamma, \alpha\delta \in I$.
\end{itemize}
Moreover, set $O_6=\{s(\alpha_x), t(\gamma_x) | x \in E_6 \}$. Therefore, if $x \in E_6$ then its neighborhood has the following shape:
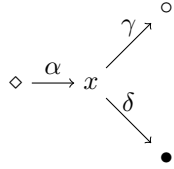
\begin{center}
	\begin{tikzpicture}
	\node  (a) at (0,0) {$\diamond$};
	\node  (b) at (2, -1) {$\bullet$};
	\node  (c) at (1, 0) {$x$};
	\node  (d) at (2, 1) {$\circ$};
	\draw [->] (a) -- node[above] { $\alpha$} (c);
	\draw [->] (c) -- node[above] { $\delta$} (b);
	\draw [->] (c) -- node[above] { $\gamma$} (d);
	\end{tikzpicture}
\end{center}

Furthermore, the elements of  $\cup_{i=1}^{6}E_i$ and $\cup_{i=1}^{6}O_i$ are called \emph{exceptional vertices} and \emph{ordinary vertices}, respectively. 
\end{definition}

\begin{definition}\label{defi-gqs} A quadratic string algebra $A=KQ/I$ is said to be \emph{good} (shortly, a gqs algebra) if every vertex of $Q_0$ is a gentle  or exceptional vertex. 
\end{definition}

\begin{remark}Note that any exceptional vertex is not gentle. In the other hand, since any ordinary vertex is a sink or a source, then it is a gentle vertex.
\end{remark}

\begin{example}(a) Every gentle algebra is a gqs algebra.\\
	
(b) Let $A=KQ/I$ given by
	
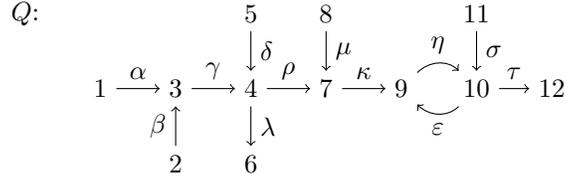
\begin{center}
		\begin{tikzpicture}
		\node  (q) at (-1, 1) {$Q$:};
		\node  (1) at (0, 0) {$1$};
		\node  (2) at (1, -1) {$2$};
		\node  (3) at (1, 0) {$3$};
		\node  (4) at (2, 0) {$4$};
		\node  (5) at (2, 1) {$5$};
		\node  (6) at (2, -1) {$6$};
		\node  (7) at (3, 0) {$7$};
		\node  (8) at (3, 1) {$8$};
		\node  (9) at (4, 0) {$9$};
		\node  (10) at (5, 0) {$10$};
		\node  (11)	at (5, 1) {$11$};
		\node  (12)	at (6, 0) {$12$};
		\draw [->] (1) -- node[above] { $\alpha$}(3);
		\draw [->] (2) -- node[left] { $\beta$}(3);
		\draw [->] (3) -- node[above] { $\gamma$}(4);
		\draw [->] (5) -- node[right] { $\delta$}(4);
		\draw [->] (4) -- node[right] { $\lambda$}(6);
		\draw [->] (4) -- node[above] { $\rho$}(7);
		\draw [->] (8) -- node[right] { $\mu$}(7);
		\draw [->] (7) -- node[above] { $\kappa$}(9);
		\draw [->] (9) to [out=45,in=135] node[above] { $\eta$} (10);
		\draw [->] (10) to [out=225,in=315] node[below] { $\varepsilon$} (9);
		\draw [->] (11) -- node[right] { $\sigma$}(10);
		\draw [->] (10) -- node[above] { $\tau$}(12);	
		\end{tikzpicture}
	\end{center} 
and $I=\langle\alpha\gamma, \beta\gamma, \gamma\lambda, \delta\lambda, \delta\rho, \rho\kappa,\varepsilon\eta,\eta\tau,\sigma\tau, \sigma\varepsilon \rangle$. Then $A$ is a gqs algebra where $E_1=\{4\}$, $O_1=\{5,6\}$, $E_2=\{10\}$, $O_2=\{11, 12\}$, $E_4=\{3\}$ and $O_4=\{1, 2\}$. 
\end{example}

In the order to show that gqs algebras are derived tame, let us recall the definition of skewed-gentle algebras.

\begin{definition}
Let $A=KQ/I$ be a quadratic string algebra. A vertex $x \in Q_0$  is said to be \emph{special} if there is at most one arrow $\alpha$ starting at $x$, at most one arrow $\beta$ ending at $x$ and if both exist then $\alpha\beta \notin I$.  
\end{definition}

Following \cite{MR3157316} (see also \cite{MR1738420,MR1986209}), a basic algebra $B$ is said to be \emph{skewed-gentle} if $B\cong A[D]$ where $A=KQ/I$ is a gentle algebra and $D \subset Q_0$ is  set of special vertices. 

It is well-known that gentle (\cite{MR1102821,MR2000963}) and skewed-gentle algebras (\cite{MR1738420,MR1986209}) are derived tame. 

\begin{lemma}\label{lemma4}
Let $A=KQ/I$ be a gqs algebra with $n> 0$ exceptional vertices and let $D\subset Q_0$ be a set of special but not ordinary vertices.  Then $A[D]$ is derived equivalent to $B[S]$  where $B=K\Delta/J$ is a gqs algebra with $n-1$ exceptional vertices and  $S\subset \Delta_0$ is a set of special but not ordinary vertices. 
\end{lemma}
\begin{proof} Fix $\mathcal{I}$ a set of paths with length two such that $I=\langle\mathcal{I}\rangle$. If  $x \in Q_0$ is an exceptional vertex, then  $x \in E_i$ for any $i \in \{1, \cdots, 6\}$, by definition.\\

\noindent \underline{Case 1:} If $x \in E_1$, let $\alpha, \beta, \gamma, \delta$ be arrows as in Definition \ref{defi01} (1). We can suppose that $\alpha: 1 \rightarrow 3$, $\beta: 2 \rightarrow 3$, $\gamma: 3 \rightarrow 4$ and $\delta: 3 \rightarrow 5$. In this way, $x=3$ and the quiver $Q$ has the shape
  
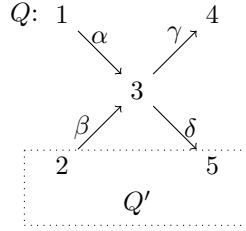
\begin{center}
	\begin{tikzpicture}
	\node  (q) at (-0.5, 1) {$Q$:};
	\node  (a) at (0, 1) {$1$};
	\node  (b) at (1, 0) {$3$};
	\node  (c) at (0, -1) {$2$};
	\node  (d) at (2, 1) {$4$};
	\node  (e) at (2, -1) {$5$};
	\node  (f) at (1,-1.5 ) {$Q'$};
	\draw [->] (a) -- node[above] { $\alpha$} (b);
	\draw [->] (c) -- node[left] { $\beta$} (b);
	\draw [->] (b) -- node[above] { $\gamma$} (d);
	\draw [->] (b) -- node[right] { $\delta$} (e);
	\draw [dotted] (-0.5,-0.8) -- (-0.5, -1.8);
	\draw [dotted] (-0.5,-0.8) -- (2.5, -0.8);
	\draw [dotted] (2.5,-0.8) -- (2.5, -1.8);
	\draw [dotted] (-0.5,-1.8) -- (2.5, -1.8);
	\end{tikzpicture}
\end{center} 
where $Q'$ is the full subquiver of $Q$ with $Q'_0=Q_0\backslash\{1, 3, 4\}$. By hypotheses $1, 4 \not \in D$. Moreover, since $3$ is an exceptional vertex, then it is not gentle, in particular, $3$ is not a special vertex. Hence $D \subset Q_0\backslash \{1,3,4\}$. 

Since $4$ is a sink and for any $d \in D$ there is no arrow $d\rightarrow 4$, then $A[D]$ is derived equivalent to $\mu_3^-(A)[D]$ by Lemma \ref{lemma3}. Moreover,  $\mu_{4}^-(A)\cong K\Omega/\langle\mathcal{I}^\ast\rangle$ where   
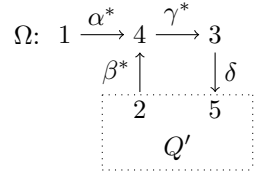
\begin{center}
	\begin{tikzpicture}
	\node  (q) at (-0.5, 0) {$\Omega$:};
	\node  (a) at (0, 0) {$1$};
	\node  (b) at (1, 0) {$4$};
	\node  (c) at (1, -1) {$2$};
	\node  (d) at (2, 0) {$3$};
	\node  (e) at (2, -1) {$5$};
	\node  (f) at (1.5,-1.5 ) {$Q'$};
	\draw [->] (a) -- node[above] { $\alpha^\ast$} (b);
	\draw [->] (c) -- node[left] { $\beta^\ast$} (b);
	\draw [->] (b) -- node[above] { $\gamma^\ast$} (d);
	\draw [->] (d) -- node[right] { $\delta$} (e);
	\draw [dotted] (0.5,-0.8) -- (0.5, -1.8);
	\draw [dotted] (0.5,-0.8) -- (2.5, -0.8);
	\draw [dotted] (2.5,-0.8) -- (2.5, -1.8);
	\draw [dotted] (0.5,-1.8) -- (2.5, -1.8);
	\end{tikzpicture}
\end{center} 
and $\mathcal{I}^\ast=(\mathcal{I} \cap KQ') \cup \{\alpha^\ast\gamma^\ast\delta\} \cup \{\mu\beta^\ast|\mu \in Q'_1, \mu\beta \in \mathcal{I}\} \cup \{\delta \mu | \mu\in Q'_1, \delta\mu \in \mathcal{I}\}$.

Again, since $1$ is a source in $\Omega$ and for any $d \in D$ there is no arrows $1 \rightarrow d$, then $\mu_{4}^-(A)[D]$ is derived equivalent to $\mu_1^+(\mu_{4}^-(A))[D]$ by Lemma \ref{lemma3}. Moreover, a straightforward calculation shows that $\mu_1^+(\mu_{4}^-(A)) \cong K\Gamma/\langle \tilde{\mathcal{I}} \rangle$ where
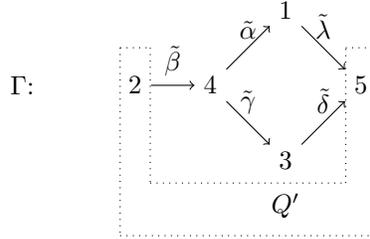
\begin{center} 
	\begin{tikzpicture}
	\node  (q) at (-1.5, 0) {$\Gamma$:};
	\node  (a) at (0, 0) {$2$};
	\node  (b) at (1, 0) {$4$};
	\node  (c) at (2, 1) {$1$};
	\node  (d) at (2, -1) {$3$};
	\node  (e) at (3, 0) {$5$};
	\node  (f) at (2,-1.6 ) {$Q'$};
	\draw [->] (a) -- node[above] { $\tilde{\beta}$} (b);
	\draw [->] (b) -- node[above] { $\tilde{\alpha}$} (c);
	\draw [->] (b) -- node[above] { $\tilde{\gamma}$} (d);
	\draw [->] (c) -- node[above] { $\tilde{\lambda}$} (e);
	\draw [->] (d) -- node[above] { $\tilde{\delta}$} (e);
	\draw [dotted] (-0.2,0.5) -- (0.2, 0.5);
	\draw [dotted] (-0.2,0.5) -- (-0.2, -2);
	\draw [dotted] (0.2,0.5) -- (0.2, -1.3);
	\draw [dotted] (0.2, -1.3) -- (2.8, -1.3);
	\draw [dotted] (-0.2,-2) -- (3.2, -2);
	\draw [dotted] (2.8, -1.3) -- (2.8, 0.5);
	\draw [dotted] (3.2, -2) -- (3.2, 0.5);
	\draw [dotted] (2.8, 0.5) -- (3.2, 0.5);
	\end{tikzpicture}
\end{center}	
and  $\tilde{\mathcal{I}}= (\mathcal{I}\cap KQ') \cup \{\tilde{\alpha}\tilde{\lambda} - \tilde{\gamma} \tilde{\delta}\} \cup\{\mu \tilde{\beta}|\mu\in Q'_1, \mu\beta \in \mathcal{I}\} \cup \{\tilde{\delta} \mu, \tilde{\lambda} \mu| \mu \in Q'_1, \delta \mu \in \mathcal{I} \}$.

Now, let $\Delta$ be the full subquiver of $\Gamma$ with $\Delta_0=\Gamma_0\backslash\{1\} $ and let $J$ be the ideal of $K\Delta$ generated by $K\Delta\cap \tilde{\mathcal{I}}$.
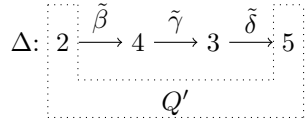
\begin{center}
	\begin{tikzpicture}
	\node  (q) at (-0.5, 0) {$\Delta$:};
	\node  (a) at (0, 0) {$2$};
	\node  (b) at (1, 0) {$4$};
	\node  (d) at (2, 0) {$3$};
	\node  (e) at (3, 0) {$5$};
	\node  (f) at (1.5,-0.8 ) {$Q'$};
	\draw [->] (a) -- node[above] { $\tilde{\beta}$} (b);
	\draw [->] (b) -- node[above] { $\tilde{\gamma}$} (d);
	\draw [->] (d) -- node[above] { $\tilde{\delta}$} (e);
	\draw [dotted] (-0.2,0.5) -- (0.2, 0.5);
	\draw [dotted] (-0.2,0.5) -- (-0.2, -1);
	\draw [dotted] (0.2,0.5) -- (0.2, -0.5);
	\draw [dotted] (0.2, -0.5) -- (2.8, -0.5);
	\draw [dotted] (-0.2,-1) -- (3.2, -1);
	\draw [dotted] (2.8, -0.5) -- (2.8, 0.5);
	\draw [dotted] (3.2, -1) -- (3.2, 0.5);
	\draw [dotted] (2.8, 0.5) -- (3.2, 0.5);
	\end{tikzpicture}
\end{center}

Clearly $B=K\Delta/J$ is a gqs with $n-1$ exceptional vertices and $S=\{3\}\cup D \subset \Delta_0$ is a set of special but not ordinary vertices. Since  $K\Gamma/\langle\tilde{\mathcal{I}}\rangle\cong B[3]$, then $\mu_1^+(\mu_{4}^-(A))[D]\cong B[S]$.\\

\noindent \underline{Case 2:} If $x \in E_2$, let $\alpha, \beta, \gamma, \delta$ be arrows as in Definition \ref{defi01} (2). We can suppose that $\alpha: 1 \rightarrow 3$, $\beta: 2 \rightarrow 3$, $\gamma: 3 \rightarrow 4$ and $\delta: 3 \rightarrow 2$. Note that this case is completely analogous to the previous case if we identify the vertices $2$ and $5$. In this way, $x=3$ and the quiver $Q$ has the shape
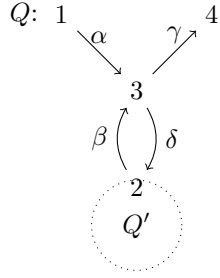
\begin{center}
	\begin{tikzpicture}
	\node  (q) at (-0.5, 1) {$Q$:};
	\node  (a) at (0, 1) {$1$};
	\node  (b) at (1, 0) {$3$};
	\node  (c) at (1, -1.3) {$2$};
	\node  (d) at (2, 1) {$4$};
	\draw [->] (a) -- node[above] { $\alpha$} (b);
	\draw [->] (c) to [bend left] node[left] { $\beta$} (b);
	\draw [->] (b) -- node[above] { $\gamma$} (d);
	\draw [->] (b) to [bend left]  node[right] { $\delta$} (c);
	\draw[dotted] (1,-1.8) node {$Q'$} (1,-1.8) circle (0.6);	
\end{tikzpicture}
\end{center}
where $Q'$ is the full subquiver of $Q$ with $Q'_0=Q_0\backslash\{1, 3, 4\}$. By hypotheses we have that $D \subset Q_0\backslash \{1,3,4\}$.

Proceeding as in the case (1), we have that $A[D]$ is derived equivalent to  $\mu_1^+(\mu_{4}^-(A)) \cong K\Gamma/\langle\tilde{I}\rangle$ where
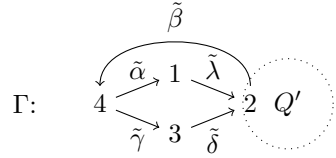
\[
\begin{tikzpicture}
\node  (q) at (0, 0) {$\Gamma$:};
\node  (b) at (1, 0) {$4$};
\node  (c) at (2, 0.4) {$1$};
\node  (d) at (2, -0.4) {$3$};
\node  (e) at (3, 0) {$2$};
\draw [->] (e) to [out =90, in =90] node [above] {$\tilde{\beta}$} (b);
\draw [->] (b) -- node[above] { $\tilde{\alpha}$} (c);
\draw [->] (b) -- node[below] { $\tilde{\gamma}$} (d);
\draw [->] (c) -- node[above] { $\tilde{\tiny{\lambda}}$} (e);
\draw [->] (d) -- node[below] { $\tilde{\delta}$} (e);
\draw[dotted] (3.5,0) node {$Q'$} (3.5,0) circle (0.6);
\end{tikzpicture}
\]
and $\tilde{\mathcal{I}}= (\mathcal{I}\cap K Q') \cup \{\tilde{\alpha}\tilde{\lambda} - \tilde{\gamma} \tilde{\delta}\}\cup \{\tilde{\delta}\tilde{\beta}, \tilde{\lambda}\tilde{\beta}\} \cup\{\mu \tilde{\beta}|\mu\in Q'_1, \mu\beta \in \mathcal{I}\} \cup \{\tilde{\delta} \mu, \tilde{\lambda} \mu| \mu \in Q'_1, \delta \mu \in \mathcal{I} \}$.

Let $B=K\Delta/J$ be the algebra where $\Delta$ is the full subquiver of $\Gamma$ such that $\Delta_0=Q_0\backslash\{1\}$ and let $J$ be the ideal of $K\Delta$ generated by $\tilde{\mathcal{I}}\cap K\Delta$. It is clear that $B$ is a gqs algebra with $n-1$ exceptional vertices, $S=\{3\}\cup D \subset \Delta_0$ is a set of special but not ordinary vertices and $K\Gamma/\langle\tilde{\mathcal{I}}\rangle\cong B[3]$. Hence $\mu_1^+(\mu_{4}^-(A))[D]\cong B[S]$.\\

\noindent \underline{Case 3:} If $x \in E_3$, let $\alpha, \beta, \gamma$ be arrows as in Definition \ref{defi01} (3). We may assume without loss of generality that $\alpha: 1 \rightarrow 3$, $\beta: 2 \rightarrow 3$ and $\gamma: 3 \rightarrow 4$. Thus, $x=3$ and the quiver $Q$ has the shape
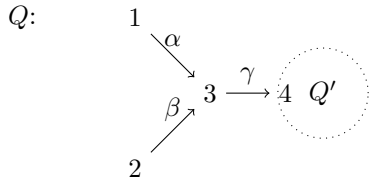
\begin{center}
	\begin{tikzpicture}
	\node  (q) at (-1.5, 1) {$Q$:};
	\node  (a) at (0, 1) {$1$};
	\node  (b) at (0, -1) {$2$};
	\node  (c) at (1, 0) {$3$};
	\node  (d) at (2, 0) {$4$};
	\draw[dotted] (2.5,0) node {$Q'$} (2.5,0) circle (0.6);
	\draw [->] (a) -- node[above] { $\alpha$} (c);
	\draw [->] (b) -- node[above] { $\beta$} (c);
	\draw [->] (c) -- node[above] { $\gamma$} (d);
	\end{tikzpicture}
\end{center}
where $x=3$ and $Q'$ is the full subquiver of $Q$ with $Q'_0=Q_0\backslash \{1, 2, 3\}$. Note that $D \subset Q_0\backslash \{1, 2, 3\}$ by hypotheses.

Let $\Delta$ be the full subquiver of $Q$ such that $\Delta_0=Q_0\backslash\{2\}$ and let $J$ be the ideal of $K\Delta$ generated by $\mathcal{I}\backslash \{\beta\gamma\}$. Clearly $B=K\Delta/J$ is a gqs algebra with $n-1$ exceptional vertices, $S=\{3\}\cup D \subset \Delta_0$ is a set of special but not ordinary vertices and $A= B[1]$. Therefore  $A[D]\cong B[S]$.\\

\indent \underline{Case 4:} If $x \in E_4$, let $\alpha, \beta, \gamma$ be arrows as in  Definition \ref{defi01} (4). We may assume without loss of generality that $\alpha: 1 \rightarrow 3$, $\beta: 2 \rightarrow 3$ and $\gamma: 3 \rightarrow 4$. Thus $x=3$, $\alpha\gamma, \beta\gamma\in I$ and the quiver $Q$ has the shape
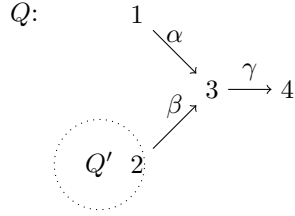
\begin{center}
	\begin{tikzpicture}
	\node  (q) at (-1.5, 1) {$Q$:};
	\node  (a) at (0, 1) {$1$};
	\node  (b) at (0, -1) {$2$};
	\node  (c) at (1, 0) {$3$};
	\node  (d) at (2, 0) {$4$};
	\draw[dotted] (-0.5,-1) node {$Q'$} (-0.5,-1) circle (0.6);
	\draw [->] (a) -- node[above] { $\alpha$} (c);
	\draw [->] (b) -- node[above] { $\beta$} (c);
	\draw [->] (c) -- node[above] { $\gamma$} (d);
	\end{tikzpicture}
\end{center}
where  and $Q'$ is the full subquiver of $Q$ with $Q'_0=Q_0\backslash\{1, 3, 4\}$. Note that $D \subset Q_0\backslash \{1,3,4\}$ by hypotheses. 

Since $4$ is a sink and for any $d \in D$ there is no arrow $d \rightarrow 4$, then $A[D]$ is derived equivalent to $\mu_4^-(A)[D]$ by Lemma \ref{lemma3}. Moreover, a straightforward calculation shows that  $\mu_{4}^-(A)\cong K\Omega/\langle \mathcal{I}^\ast \rangle$ where
	
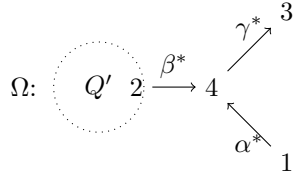
\begin{center}
	\begin{tikzpicture}
	\node  (q) at (-1.5, 0) {$\Omega$:};
	\node  (a) at (2, 1) {$3$};
	\node  (b) at (0, 0) {$2$};
	\node  (c) at (1, 0) {$4$};
	\node  (d) at (2, -1) {$1$ };
	\draw[dotted] (-0.5,0) node {$Q'$} (-0.5,0) circle (0.6);
	\draw [->] (c) -- node[above] { $\gamma^\ast$} (a);
	\draw [->] (b) -- node[above] { $\beta^\ast$} (c);
	\draw [->] (d) -- node[below] { $\alpha^\ast$} (c);
	\end{tikzpicture}
\end{center}
and $\mathcal{I}^\ast=(\mathcal{I} \cap K Q') \cup \{\mu\beta^\ast|\mu \in Q'_1, \mu\beta \in \mathcal{I}\}$.

Since $1$ is source and for any $d \in D$ there is no arrow $1 \rightarrow d$, then $\mu_{4}^-(A)[D]$ is derived equivalent to $\mu_1^+(\mu_{4}^-(A))[D]$ by Lemma \ref{lemma3}. Furthermore, $\mu_1^+(\mu_{4}^-(A)) \cong K\Gamma/\langle \mathcal{I}^\ast\rangle$ where $\Gamma$ is the  quiver
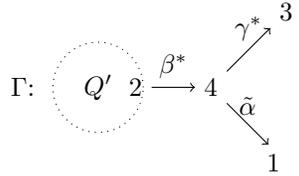
\begin{center}
	\begin{tikzpicture}
	\node  (q) at (-1.5, 0) {$\Gamma$:};
	\node  (a) at (2, 1) {$3$};
	\node  (b) at (0, 0) {$2$};
	\node  (c) at (1, 0) {$4$};
	\node  (d) at (2, -1) {$1$ \ .};
	\draw[dotted] (-0.5,0) node {$Q'$} (-0.5,0) circle (0.6);
	\draw [->] (c) -- node[above] { $\gamma^\ast$} (a);
	\draw [->] (b) -- node[above] { $\beta^\ast$} (c);
	\draw [->] (c) -- node[above] { $\tilde{\alpha}$} (d);
	\end{tikzpicture}
\end{center}

Let $\Delta$ be the full subquiver of $\Gamma$ such that $\Delta_0=\Gamma_0\backslash\{1\} $ and let $J$ be the ideal of $K\Delta$ generated by $K\Delta\cap \mathcal{I}^\ast$.
Obviously $B=K\Delta/J$ is a gqs algebra with $n-1$ exceptional vertices, $S=\{3\}\cup D\subset \Delta_0$ is a set of special but not ordinary vertices and $K\Gamma/L=B[3]$.  Therefore $\mu_1^+(\mu_{4}^-(A))[D]\cong B[S]$.\\

Finally, if $x \in E_5$ or $x \in E_6$, then we proceed in a similar way as in cases 3 and 4, respectively. Thus we have that $A[D]$ is derived equivalent to $B[S]$, where $B$ is a gqs algebra with $n-1$ exceptional vertices and $S$ is a set of special but not ordinary vertices of $B$.
\end{proof}

\begin{prop}\label{prop1}
Let $A=KQ/I$ be a gqs algebra and let $D \subset Q_0$ be a set of special but not ordinary vertices. Then $A[D]$ is derived equivalent to some skewed-gentle algebra. In particular, $A[D]$ is derived tame.
\end{prop}

\begin{proof}
We proceed by induction on the number of exceptional vertices of $A$. If $A$ does not have exceptional vertices, then $A$ is a gentle algebra and $A[D]$ is a skewed-gentle algebra by definition.

If $A$ has $n>0$ exceptional vertices, then by Lemma \ref{lemma4} we have that $A[D]$ is derived equivalent to $B[S]$, where $B$ is a gqs algebra with $n-1$ exceptional vertices and $S$ is a set of special but not ordinary vertices of $B$. Hence, $B[S]$ is derived equivalent to some skewed-gentle algebra by induction hypotheses.       
\end{proof}

\noindent\textbf{Proof of Theorem A.} For any gqs algebra $A$, we have that $A= A[D]$ where $D=\emptyset$.  Therefore,  the result follows by  Proposition \ref{prop1}. \hspace*{4.7cm} $\square$   

\section{Quadratic String Algebras}

The objective of this section is to prove the Theorems B and D.

\begin{lemma}\label{lemma5}
Let $A=KQ/I$ be a monomial algebra and suppose that there are arrows $\alpha, \beta, \gamma \in Q_1$, $\alpha\not=\beta$, such that  $s(\alpha)=t(\gamma)$, $t(\alpha)=s(\gamma)=t(\beta)$ and  $\alpha\gamma, \gamma\alpha, \beta\gamma \in I$ (resp. $s(\alpha)=s(\beta)=t(\gamma)$, $t(\alpha)=s(\gamma)$ and  $\alpha\gamma, \gamma\alpha, \gamma\beta \in I$). Then $A$ is derived wild.
\end{lemma}

\begin{proof} We assume that $s(\alpha)=t(\gamma)$, $t(\alpha)=s(\gamma)=t(\beta)$ and $\alpha\gamma, \gamma\alpha, \beta\gamma \in I$ (the other case is similar). By construction,  the Galois covering  of $A$ (induced by fundamental group $\pi_1(Q, I)$) has a full subcategory $\widehat{A}\cong K\widehat{Q}/\mathcal{R}_{\widehat{Q}}^2$, where $\widehat{Q}$ has the following shape:
	\begin{center}
		\begin{tikzpicture}
		\node  (q) at (-0.5, 0) {$\widehat{Q}$:};
		\node  (a) at (0, 0) {$\bullet$};
		\node  (b) at (1, 0) {$\bullet$};
		\node  (c) at (2, 0) {$\bullet$};
		\node  (d) at (3, 0) {$\bullet$};
		\node  (e) at (4, 0) {$\bullet$};
		\node  (f) at (5, 0) {$\bullet$};
		\node  (g) at (2, -1) {$\bullet$};
		\node  (h) at (4, -1) {$\bullet$};
		\draw [->] (a) -- node[above] { $\gamma$} (b);
		\draw [->] (b) -- node[above] { $\alpha$} (c);
		\draw [->] (c) -- node[above] { $\gamma$} (d);
		\draw [->] (d) -- node[above] { $\alpha$} (e);
		\draw [->] (e) -- node[above] { $\gamma$} (f);
		\draw [->] (g) -- node[right] { $\beta$} (c);
		\draw [->] (h) -- node[right] { $\beta$} (e);
		\end{tikzpicture}
	\end{center}
Therefore, $A$ is derived wild  by  Lemma \ref{lemma1}.   
\end{proof}

\begin{lemma}\label{lemma6}
Let $A=KQ/I$ be a derived tame monomial algebra and suppose there are arrows $\alpha, \beta, \gamma \in Q_1$, $\alpha \not=\beta$, such that $t(\alpha)=t(\beta)=s(\gamma)$ and $\alpha \gamma, \beta \gamma \in I$ $($resp. $s(\alpha)=s(\beta)=t(\gamma)$ and $\gamma \alpha, \gamma \beta \in I)$. Then $\alpha\not=\gamma$, $\beta\not=\gamma$ and the vertices  $s(\alpha), s(\beta), t(\alpha)$ and $t(\gamma)$ are pairwise distinct.     
\end{lemma}

\begin{proof}
We assume that $t(\alpha)=t(\beta)=s(\gamma)$ and $\alpha \gamma, \beta \gamma \in I$ (the other case is similar).

If $\alpha=\gamma$, then $s(\alpha)=t(\alpha)=t(\beta)$ and $\alpha^2, \beta\alpha \in I$, and so $A$ is derived wild by Lemma \ref{lemma5}, a contradiction. Hence $\alpha\not=\gamma$. In the same way $\beta\not=\gamma$.

Suppose that $s(\alpha)=t(\alpha)$. Since $I$ is admissible, there is $n\geq 2$ such that $\alpha^{n-1}\notin I$ but $\alpha^{n} \in I$. Putting $h=\alpha^{n-1}$, the Galois covering of $A$ has a full subcategory $\widehat{A}\cong K\widehat{Q}/\mathcal{R}_{\widehat{Q}}^2$, where $\widehat{Q}$ has the following shape:
\begin{center}
	\begin{tikzpicture}
	\node  (q) at (-0.5, 0) {$\widehat{Q}$:};
	\node  (a) at (0, 0) {$\bullet$};
	\node  (b) at (1, 0) {$\bullet$};
	\node  (c) at (2, 0) {$\bullet$};
	\node  (d) at (3, 0) {$\bullet$};
	\node  (e) at (4, 0) {$\bullet$};
	\node  (f) at (1, -1) {$\bullet$};
	\node  (g) at (2, -1) {$\bullet$};
	\draw [->] (a) -- node[above] { $h$} (b);
	\draw [->] (b) -- node[above] { $h$} (c);
	\draw [->] (c) -- node[above] { $h$} (d);
	\draw [->] (d) -- node[above] { $h$} (e);
	\draw [->] (b) -- node[right] { $\gamma$} (f);
	\draw [->] (c) -- node[right] { $\gamma$} (g);
	\end{tikzpicture}
\end{center}
Hence  $A$ is derived wild by Lemma \ref{lemma1}, a contradiction. Thus $s(\alpha)\not=t(\gamma)$. Similarly we have that $s(\beta)\not=t(\beta)=t(\alpha)$ and $s(\alpha)=s(\delta)\not=t(\delta)$.  

If $s(\alpha)=t(\gamma)$, then  $\gamma\alpha \notin I$ by Lemma \ref{lemma5}. Putting $h=\gamma\alpha$, the Galois covering of $A$ has a full subcategory $\widehat{A}\cong K\widehat{Q}/\mathcal{R}_{\widehat{Q}}^2$ given by following quiver: 
	\begin{center}
		\begin{tikzpicture}
		\node  (q) at (-0.5, 0) {$\widehat{Q}$:};
		\node  (a) at (0, 0) {$\bullet$};
		\node  (b) at (1, 0) {$\bullet$};
		\node  (c) at (2, 0) {$\bullet$};
		\node  (d) at (3, 0) {$\bullet$};
		\node  (e) at (4, 0) {$\bullet$};
		\node  (g) at (1, -1) {$\bullet$};
		\node  (h) at (2, -1) {$\bullet$};
		\draw [->] (a) -- node[above] { $h$} (b);
		\draw [->] (b) -- node[above] { $h$} (c);
		\draw [->] (c) -- node[above] { $h$} (d);
		\draw [->] (d) -- node[above] { $h$} (e);
		\draw [->] (g) -- node[right] { $\beta$} (b);
		\draw [->] (h) -- node[right] { $\beta$} (c);
		\end{tikzpicture}
	\end{center}
Hence $A$ is derived wild by Lemma \ref{lemma1}, a contradiction. Therefore $s(\alpha)\not=t(\gamma)$. Similarly we have that $s(\beta)\not=t(\gamma)$.	 

Finally, if $s(\alpha)=s(\beta)$, then the Galois covering of $A$ has a full subcategory $\widehat{A}\cong K\widehat{Q}/\mathcal{R}_{\widehat{Q}}^2$, where:
\begin{center}
	\begin{tikzpicture}
	\node  (q) at (-0.5, 0) {$\widehat{Q}$:};
	\node  (a) at (0, 0) {$\bullet$};
	\node  (b) at (1, 0) {$\bullet$};
	\node  (c) at (2, 0) {$\bullet$};
	\node  (d) at (3, 0) {$\bullet$};
	\node  (e) at (4, 0) {$\bullet$};
	\node  (f) at (5, 0) {$\bullet$};
	\node  (g) at (1, -1) {$\bullet$};
	\node  (h) at (3, -1) {$\bullet$};
	\draw [->] (a) -- node[above] { $\beta$} (b);
	\draw [->] (c) -- node[above] { $\alpha$} (b);
	\draw [->] (c) -- node[above] { $\beta$} (d);
	\draw [->] (e) -- node[above] { $\alpha$} (d);
	\draw [->] (e) -- node[above] { $\beta$} (f);
	\draw [->] (b) -- node[right] { $\gamma$} (g);
	\draw [->] (d) -- node[right] { $\gamma$} (h);
	\end{tikzpicture}
\end{center}
Again $A$ is derived wild by Lemma \ref{lemma1}, a contradiction.  Therefore $s(\alpha)\not=s(\beta)$.	
\end{proof}	

\begin{lemma}\label{lemma7} Let $A= KQ/I$ be a derived tame quadratic string algebra and suppose that there are arrows $\alpha, \beta, \gamma, \delta \in Q_1$, $\alpha \not=\beta$, $\gamma\not=\delta$ such that $t(\alpha)=t(\beta)=s(\gamma)=s(\delta)$ and $\alpha\gamma, \alpha\delta, \beta\gamma, \beta\delta \in I$. Then $Q$ has the following shape.
	\begin{center}
		\begin{tikzpicture}
		\node  (q) at (-0.5, 1) {$Q$:};
		\node  (a) at (0, 1) {$\bullet$};
		\node  (b) at (0, -1) {$\bullet$};
		\node  (c) at (1, 0) {$\bullet$};
		\node  (d) at (2, 1) {$\bullet$};
		\node  (e) at (2, -1) {$\bullet$};
		\draw [->] (a) -- node[above] { $\alpha$} (c);
		\draw [->] (b) -- node[above] { $\beta$} (c);
		\draw [->] (c) -- node[above] { $\gamma$} (d);
		\draw [->] (c) -- node[above] { $\delta$} (e);
		\end{tikzpicture}
	\end{center}
In particular, $A$ is a tree.
\end{lemma}

\begin{proof} Note that the arrows $\alpha, \beta, \gamma$ and $\delta$ are pairwise distinct and the vertices $s(\alpha), s(\beta), t(\alpha), t(\gamma)$ and $t(\delta)$ are also pairwise distinct by Lemma \ref{lemma6}.
	 	
Suppose that $\{\alpha, \beta, \gamma, \delta\} \not= Q_1$. Since $Q$ is connected quiver, then there is an arrow $\lambda\in Q_1\backslash \{\alpha, \beta, \gamma,\delta\}$ such that $s(\lambda)$ or $t(\lambda)$ belongs to $\{s(\alpha), s(\beta), t(\gamma), t(\delta)\}$.
	
If $s(\lambda)= s(\alpha)$ or $t(\lambda)= s(\alpha)$, then the Galois covering of $A$ has a full subcategory  $\tilde{A}$ which is  isomorphic to the path algebra  of some following  quivers wiht relations:
\begin{flushleft}
	\begin{tabular}{l}
		\begin{tikzpicture}
		\node  (q) at (-2, 1) {$(1)$};
		\node  (q) at (-1, 1) {$\Delta$:};
		\node  (x) at (0, 1) {$\bullet$};
		\node  (a) at (1, 1) {$\bullet$};
		\node  (b) at (1,-1 ) {$\bullet$};
		\node  (c) at (2, 0) {$\bullet$};
		\node  (d) at (3, 1) {$\bullet$};
		\node  (e) at (3, -1) {$\bullet$};
		\node  (i) at (6, 1) { and \ \ $J=\mathcal{R}_{\Delta}^2$;};
		\draw [->] (a) -- node[above] { $\alpha$} (c);
		\draw [->] (b) -- node[above] { $\beta$} (c);
		\draw [->] (c) -- node[above] { $\gamma$} (d);
		\draw [->] (c) -- node[above] { $\delta$} (e);
		\draw [->] (a) -- node[above] { $\lambda$} (x);
		\end{tikzpicture}
	\end{tabular}
\end{flushleft}	

\begin{flushleft}
	\begin{tabular}{l}
		
		\begin{tikzpicture}
		\node  (q) at (-2, 1) {(2)};
		\node  (q) at (-1, 1) {$\Gamma$:};
		\node  (x) at (0, 1) {$\bullet$};
		\node  (a) at (1, 1) {$\bullet$};
		\node  (b) at (1,-1 ) {$\bullet$};
		\node  (c) at (2, 0) {$\bullet$};
		\node  (d) at (3, 1) {$\bullet$};
		\node  (e) at (3, -1) {$\bullet$};
		\node  (i) at (6, 1) { and \ \  $H=\mathcal{R}_{\Gamma}^2$;};
		\draw [->] (a) -- node[above] { $\alpha$} (c);
		\draw [->] (b) -- node[above] { $\beta$} (c);
		\draw [->] (c) -- node[above] { $\gamma$} (d);
		\draw [->] (c) -- node[above] { $\delta$} (e);
		\draw [->] (x) -- node[above] { $\lambda$} (a);
		\end{tikzpicture}
	\end{tabular}	
	
	\begin{tabular}{l}
		\begin{tikzpicture}
		\node  (q) at (-2, 1) {$(3)$};
		\node  (q) at (-1, 1) {$\Omega$:};
		\node  (x) at (0, 1) {$\bullet$};
		\node  (a) at (1, 1) {$\bullet$};
		\node  (b) at (1,-1 ) {$\bullet$};
		\node  (c) at (2, 0) {$\bullet$};
		\node  (d) at (3, 1) {$\bullet$};
		\node  (e) at (3, -1) {$\bullet$};
		\node  (i) at (7, 1) { and \ \  $L=\langle \rho \in \Omega_2 | \rho \not= \lambda \alpha \rangle$.};
		\draw [->] (a) -- node[above] { $\alpha$} (c);
		\draw [->] (b) -- node[above] { $\beta$} (c);
		\draw [->] (c) -- node[above] { $\gamma$} (d);
		\draw [->] (c) -- node[above] { $\delta$} (e);
		\draw [->] (x) -- node[above] { $\lambda$} (a);
		\end{tikzpicture}
	\end{tabular}
\end{flushleft}
In the cases (1) and (2),  we have that $A$ is derived wild by Lemma \ref{lemma1}, a contradiction. If $\tilde{A}\cong K \Omega/ L$, then since $s(\lambda) \in \Omega_0$ is a source and there is no relation $\rho \in \{\rho \in \Omega_2 | \rho \not= \lambda \alpha\}$ such that $s(\rho) = s(\lambda)$, so $\tilde{A}$ is derived equivalent to $\mu_{s(\lambda)}^+(K \Omega/ L) \cong K \Delta/ J$, thus $A$ is derived wild by Lemma \ref{lemma1},  a contradiction. Therefore  $s(\lambda)\not= s(\alpha)$ and $t(\lambda)\not= s(\alpha)$. 

Similarly we have that $s(\lambda), t(\lambda) \not\in\{s(\beta), t(\gamma), t(\delta)\}$, a contradiction. Thus $Q_1=\{\alpha, \beta, \gamma, \delta\}$.     
\end{proof}

\begin{prop}\label{prop2}
Let $A=KQ/I$ be a derived tame quadratic string algebra where $Q$ is not a tree. If $x \in Q_0$ is not gentle vertex, then one of the following conditions holds:
\begin{itemize}
	\item[(1)] There are $\alpha,\beta,\gamma \in Q_1 $ such that $t(\alpha)=t(\beta)=s(\gamma)=x$; $s(\alpha), s(\beta), x$ and $ t(\gamma)$ are pairwise distinct; $x^+=\{\gamma\}$;  $\alpha\gamma, \beta\gamma \in I$.
	\item[(2)] There are $\alpha, \gamma, \delta \in Q_1$ such that $t(\alpha)=s(\gamma)=s(\delta)=x$; $s(\alpha), x, t(\gamma)$ and $ t(\delta)$ are pairwise distinct; $x^-=\{\alpha\}$; $\alpha\gamma, \alpha\delta \in I$.
	\item [(3)] There are $\alpha, \beta, \gamma, \delta \in Q_1$ such that $t(\alpha)=t(\beta)=s(\gamma)=s(\delta)=x$; $s(\alpha), s(\beta), x, t(\gamma)$ and $ t(\delta)$ are pairwise distinct; $\alpha\gamma, \alpha\delta, \beta\gamma \in I$; $\beta\delta \notin I$.
	\item[(4)] There are $\alpha, \beta, \gamma, \delta \in Q_1$ such that $t(\alpha)=t(\beta)=s(\gamma)=s(\delta)=x$; $s(\beta)=t(\delta)$; $s(\alpha), s(\beta), x$ and $ t(\gamma)$ are pairwise distinct; $\alpha\gamma, \alpha\delta, \beta\gamma, \delta\beta \in I$; $\beta\delta \notin I$.
\end{itemize}
\end{prop}
\begin{proof}
Since $x$ is not a gentle vertex, then at least one of the following situations holds:
\begin{itemize}
	\item[(i)] There are $\alpha, \beta, \gamma \in Q_1$, $\alpha \not=\beta$, such that $t(\alpha)=t(\beta)=s(\gamma)=x$ and 
	$\alpha\gamma, \beta\gamma\in I$. 
	\item[(ii)] There are $\alpha, \gamma, \delta \in Q_1$, $\gamma \not=\delta$ such that $s(\alpha)=t(\gamma)=t(\delta)=x$ and $\alpha\gamma, \alpha\delta\in I$.
\end{itemize}     

Suppose that the situation (i) holds. By Lemma \ref{lemma6} we have that $\alpha\not=\gamma, \beta\not=\gamma$ and the vertices $s(\alpha), s(\beta), x$ and $ t(\gamma)$ are pairwise distinct. If $x^+=\{\gamma\}$, then the condition $(1)$ holds. In the other hand, if there is an arrow $\delta\not=\gamma$ such that $s(\delta)=x$, then we must have  $\alpha\delta \in I$ or $\beta\delta\in I$. But, since $Q$ is not a tree, follows from Lemma \ref{lemma7} that we cannot have $\alpha\delta, \beta\delta\in I$. Hence we can assume, without loss of generality, that $\alpha\delta\in I$ but $\beta\delta \notin I$,  thus $t(\delta)\not=t(\gamma)$ and $t(\delta)\not=s(\alpha)$ by Lemma \ref{lemma6}. Moreover, if $t(\delta)\not=s(\beta)$, then the condition $(3)$ holds. But if $s(\beta)=t(\delta)$, then $\delta\beta \in I$ because $A$ has finite dimensional and $I$ is generated by paths of length 2. Therefore  the condition $(4)$ holds. 

Finally, a similar analysis to the case (ii) finish the proof. 
\end{proof}

\begin{lemma}\label{lemma8} Let $A=KQ/\mathcal{R}_Q^2$ be an algebra where

			\begin{center}
				\begin{tikzpicture}
				\node  (q) at (-0.5, 1) {$Q$:};
				\node  (10) at (0, 1) {$10$};
				\node  (9) at (1, 1) {$9$};
				\node  (8) at (2, 1) {$8$};
				\node  (7) at (3, 1) {$7$};
				\node  (6) at (4, 1) {$6$};
				\node  (1) at (5, 1) {$1$};
				\node  (2) at (5, -1){$2$};
				\node  (3) at (6, 0) {$3$};
				\node  (4) at (7, 1) {$4$};
				\node  (5) at (7, -1){$5$};
				\draw [-] (10) -- node[above] { $\omega_5$} (9);
				\draw [-] (9) -- node[above] { $\omega_4$} (8);
				\draw [-] (8) -- node[above] { $\omega_3$} (7);
				\draw [-] (7) -- node[above] { $\omega_2$} (6);
				\draw [-] (6) -- node[above] { $\omega_1$} (1);
				\draw [->] (1) -- node[above] { $\alpha$} (3);
				\draw [->] (2) -- node[above] { $\beta$} (3);
				\draw [->] (3) -- node[above] { $\gamma$} (4);
				\draw [->] (3) -- node[above] { $\delta$} (5);
				\end{tikzpicture}
			\end{center}
and the arrows without orientation can be oriented in either way. Then $A$ is derived wild.
\end{lemma}

\begin{proof} We define a bounded complex $T:=\bigoplus_{i=1}^{10}T^i$ of projective $A$-modules. For $i \in Q_0 \backslash \{3,4,5\}$, let $T^i:\xymatrix{ 0\ar[r]&  P_i\ar[r]& 0}$ be the stalk complex concentrated in degree $0$. Moreover, consider the following complexes
$$ \xymatrix{T^3:\cdots \ \ 0\ar[r] & P_4\oplus P_5 \ar[r]^-{\left[\gamma \,\, \delta \right]}&  P_3\ar[r]& 0 \ \ \cdots,\\
		T^4:\cdots \ \ 0\ar[r] & P_4 \ar[r]^\gamma&  P_3\ar[r]& 0 \ \ \cdots,\\
		T^5:\cdots \ \ 0\ar[r] & P_5\ar[r]^\delta&  P_3\ar[r]& 0 \ \ \cdots, } $$ 
concentrated in degree $0$ and $-1$. It is straightforward to check that  $T$ is a tilting complex and that $(\mathrm{End}(T))^{op}\cong K\Delta/J$  where
	\begin{center}
		\begin{tikzpicture}
		\node  (q) at (-0.5,0) {$\Delta$:};
		\node  (10) at (0, 0) {$10$};
		\node  (9) at (1, 0) {$9$};
		\node  (8) at (2, 0) {$8$};
		\node  (7) at (3, 0) {$7$};
		\node  (6) at (4, 0) {$6$};
		\node  (1) at (5, 0) {$1$};
		\node  (3) at (6, 0){$3$};
		\node  (4) at (7, 0) {$4$};
		\node  (2) at (8, 0) {$2$};
		\node  (5) at (6, -1){$5$};
		\draw [-] (10) -- node[above] { $\omega_5$} (9);
		\draw [-] (9) -- node[above] { $\omega_4$} (8);
		\draw [-] (8) -- node[above] { $\omega_3$} (7);
		\draw [-] (7) -- node[above] { $\omega_2$} (6);
		\draw [-] (6) -- node[above] { $\omega_1$} (1);
		\draw [->] (1) -- node[above] { $\alpha^\ast$} (3);
		\draw [->] (3) -- node[above] { $\gamma^\ast$} (4);
		\draw [->] (3) -- node[right] { $\delta^\ast$} (5);
		\draw [->] (2) -- node[above] { $\beta^\ast$} (4); 	
		\end{tikzpicture}
	\end{center}
and  $J= \langle \Delta_2\backslash \{\alpha^\ast\gamma^\ast,\alpha^\ast\delta^\ast\} \rangle$. 

Since the vertices $4, 5 \in \Delta_0$ are sinks and there is no relation $\rho \in \Delta_2\backslash \{  \alpha^\ast\gamma^\ast,\alpha^\ast\delta^\ast\} $ such that $t(\rho) \in \{4,5\}$,it  follows that  $K\Delta/J$ is derived equivalent to \linebreak$\mu_5^-(\mu_{4}^-(K\Delta/J))\cong K\Gamma/\mathcal{R}_\Gamma^2$ where $\Gamma$ is the following quiver:
	
\begin{center}
	\begin{tikzpicture}
	\node  (q) at (-0.5, 0) {$\Gamma$:};
	\node  (10) at (0, 0) {$10$};
	\node  (9) at (1, 0) {$9$};
	\node  (8) at (2, 0) {$8$};
	\node  (7) at (3, 0) {$7$};
	\node  (6) at (4, 0) {$6$};
	\node  (1) at (5, 0) {$1$};
	\node  (3) at (6, 0){$3$};
	\node  (4) at (7, 0) {$4$};
	\node  (2) at (8, 0) {$2$\ .};
	\node  (5) at (6, -1){$5$};
	\draw [-] (10) -- node[above] { $\omega_5$} (9);
	\draw [-] (9) -- node[above] { $\omega_4$} (8);
	\draw [-] (8) -- node[above] { $\omega_3$} (7);
	\draw [-] (7) -- node[above] { $\omega_2$} (6);
	\draw [-] (6) -- node[above] { $\omega_1$} (1);
	\draw [->] (1) -- node[above] { $\alpha^\ast$} (3);
	\draw [->] (4) -- node[above] { $\tilde{\gamma}$} (3);
	\draw [->] (5) -- node[right] { $\tilde{\delta}$} (3);
	\draw [->] (4) -- node[above] { $\tilde{\beta}$} (2); 	
	\end{tikzpicture}
\end{center}
Note that $\mu_5^-(\mu_{4}^-(K\Delta/J))$ is derived wild by Theorem \ref{rad}. Therefore $A$ is derived wild also.
\end{proof}

\begin{lemma}\label{lemma}
Let $Q$ be a connected and locally finite tree with an infinite number of vertices. For all $x \in Q_0$ and $n \geq 1$ there is a reduced walk $w$ such that $s(w)=x$ and $l(w)=n$.
\end{lemma}

\begin{proof} Fix $x \in Q_0$. For every $k \geq 0$ we define the set:	$$U_k=\{y\in Q_0 | \exists \,  w: x \rightarrow y \   \mbox{a reduced walk }, \,  l(w)=k \}.$$
Since $Q$ is a connected tree, then for every vertex $y \in Q_0$ there is only one reduced walk $w:x \rightarrow y$. In particular, $Q_0= \bigcup_{k\geq 0} U_k$ and $U_k \cap U_j=\emptyset$, if $k \not=j$.\\
	
\noindent \underline{Statement:} $U_k$ is a finite set for all $k\geq 0$.
	
We use induction on $k$. Since $U_0=\{e_x\}$, the statement holds for $k=0$. Let $k \geq 1$ and suppose that  $U_{k-1}$ is finite.  For every $y \in Q_0$ we define:
	\[V_y=\{z \in Q_0 | \exists \, \alpha \in Q_1, \alpha:y \rightarrow z \mbox{ or } \alpha: z \rightarrow y \}.
	\] 
It is clear that if $z \in U_{k}$, then $z \in V_y$ for some $y \in U_{k-1}$. Thus $U_{k} \subseteq \bigcup_{y \in U_{k-1}}V_y $. Moreover, since $Q$ is locally finite then $V_y$ is finite for all $y \in Q_0$. Hence $\bigcup_{y \in U_{k-1}}V_y$ and $U_{k} $ are finite sets. This proves the statement. \\
	
Now, suppose there is $k_0 \geq 0$ such that $U_{k_0}= \emptyset$, so $U_k = \emptyset $ if $k \geq k_0$. Since $Q_0 = \bigcup_{k\geq 0} U_k= \bigcup_{0 \leq i <  k_0 } U_i$ is an infinite set, then there is  $0 \leq k < k_0$ such that $U_k$ is an infinite set, a contradiction. Hence $U_k \not= \emptyset $ for any $k \geq 0$.
	
Finally, given $n\geq 1$, take $y \in U_n$. Thus, there is a reduced walk $w: x \rightarrow y$ with $l(w)=n$. 
\end{proof}

\begin{lemma}\label{lema}
Let $Q$ be a connected and locally finite tree with infinite vertices, and let $I$ be an admissible ideal of $KQ$. Suppose that there is $m \geq 2$ such that $\mathcal{R}_Q^m \subseteq I $. Then for all vertex $x$ of $Q$ and $n \geq 1$, there is a set of vertices \linebreak$\{x=x_0, x_1, \cdots, x_n \} \subset Q_0$ such that they define a full subcategory $B$ of $A= K Q/I$, such that $B\cong K \Gamma/\mathcal{R}_{\Gamma}^2$ where the underlying graph of $\Gamma$ has the following shape: 
	\[\overline{\Gamma}: \xymatrix{ x_0 \ar@{-}[r] & x_1 \ar@{-}[r] & \cdots \ar@{-}[r] & x_{n-1}  \ar@{-}[r] & x_n.  } \]
\end{lemma}

\begin{proof}
Given $x \in Q_0$ and $n \geq 1$, then by Lemma \ref{lemma} there is a reduced walk $w$ such that  $s(w)=x$ and $l(w)=mn$. Let $\{y_0=x, y_1, \cdots, y_{mn} \} \subset Q_0$ be the the set of vertices witch occur on $w$, thus we can write   $w= \alpha_0^{\epsilon_{0}}\alpha_1^{\epsilon_{1}}\cdots \alpha_{nm-1}^{\epsilon_{nm-1}}$, where $\alpha_i \in Q_1$, $\epsilon_{i} \in \{1, -1\}$,  $s(\alpha_{i}^{\epsilon_{i}})=y_i$ and $t(\alpha_{i}^{\epsilon_{i}})=y_{i+1}$ for all $0\leq i \leq nm-1$.
	\[w= \xymatrix{y_0\ar@{-}[r]^{\alpha_0} & y_1 \ar@{-}[r]^{\alpha_1} & y_2\ar@{-}[r]^{\alpha_2} & \cdots \ar@{-}[r]^{\alpha_{nm-2}} & y_{nm-1} \ar@{-}[r]^{\alpha_{nm-1}} & y_{nm}}\] 
	
Putting $l_0=0$, we define $l_{i+1}$ as follows:
\begin{itemize}
\item If $\epsilon_{l_i}=1$, then $\alpha_{l_i}:y_{l_i} \rightarrow y_{l_{i}+1} $. In this case, $l_{i+1}$ is the bigger integer such that the path $\alpha_{l_i}\alpha_{l_i+1} \cdots \alpha_{l_{i+1}-1}:y_{l_i}\rightarrow y_{l_{i+1}}$ is defined and it does not belong to $I$; 
\item If $\epsilon_{l_i}=-1$, then $\alpha_{l_i}:y_{l_{i}+1} \rightarrow y_{l_{i}} $. In this case, $l_{i+1}$ is the bigger integer such that the path $\alpha_{l_{i+1}-1}\cdots\alpha_{l_i+1} \alpha_{l_{i}}:y_{l_{i+1}}\rightarrow y_{l_{i}}$ is defined and it does not belong to $I$. 
\end{itemize}
	
Since every path of $Q$ whose length is greater than  $m-1$ belongs to $I$, then for any $i$ we have $l_{i+1}-l_i< m$. Thus we have:
	\begin{eqnarray*}
		l_n & = & l_n + (-l_{n-1} + l_{n-1}) + (-l_{n-2}+ l_{n-2}) +  \cdots + (-l_1 + l_1) + (-l_0 + l_0)\\
		& = & (l_n - l_{n-1}) + (l_{n-1} - l_{n-2}) + \cdots + (l_1 - l_0) + l_0\\
		& < & nm. 
	\end{eqnarray*}
Hence we can define $x_i:= y_{l_i}$ for all $0\leq i \leq n $. Moreover, for $0 \leq i \leq n-1$ we define the path $w_i$ as follows:
\[\left\{ \begin{array}{cl} w_i= \alpha_{l_i} \cdots \alpha_{l_{i+1}-1}: x_i \rightarrow x_{i+1}, & \mbox{ if } \epsilon_{l_i}=1;\\
w_i= \alpha_{l_{i+1}-1} \cdots \alpha_{l_{i}}: x_{i+1} \rightarrow x_{i}, & \mbox{ if } \epsilon_{l_i}=-1.
\end{array}
\right.\]
Note that for any $0 \leq i \leq n-1$, the paths $w_iw_{i+1}$ and $w_{i+1}w_i$ are not defined or they belong to $I$. 
	
Let $B$ be the full subcategory of $A$ defined by $\{x_0, \cdots, x_n\}$. Hence $B\cong K \Gamma/ \mathcal{R}_{\Gamma}^2$ where the underline graph of $\Gamma$ is given by
	\[\overline{\Gamma}: \xymatrix{ x=x_0 \ar@{-}[r]^{w_0} & x_1 \ar@{-}[r]^{w_1} & \cdots \ar@{-}[r]^{w_{n-2}} & x_{n-1}  \ar@{-}[r]^{w_{n-1}} & x_n.  } \] 
\end{proof}

\begin{lemma}\label{lemma13}
Let $A=K Q/I$ be a derived tame monomial algebra. Suppose that the Galois covering $K\widehat{Q}/\widehat{I}$ of $A$ has a full subcategory $ K\tilde{Q}/ \tilde{I}$ where $\tilde{Q}$ has the following shape
	
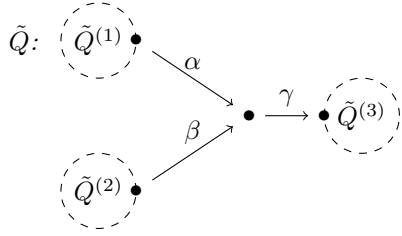
\begin{center}
		\begin{tikzpicture}
		\node  (x) at (1, 1) {$\tilde{Q}$:};
		\node  (a) at (2.5, 1) {$\bullet$};
		\node  (b) at (2.5, -1) {$\bullet$};
		\node  (c) at (4, 0) {$\bullet$};
		\node  (d) at (5, 0) {$\bullet$};
		\draw [->] (a) -- node[above] { $\alpha$}  (c);
		\draw [->] (b) -- node[above] { $\beta$} (c) ;
		\draw [->] (c) -- node[above] { $\gamma$} (d);
		\draw[dashed] (2,-1) node {$\tilde{Q}^{(2)}$} (2,-1) circle (0.5);
		\draw[dashed] (2,1) node {$\tilde{Q}^{(1)}$} (2,1) circle (0.5);
		\draw[dashed] (5.5,0) node {$\tilde{Q}^{(3)}$} (5.5,0) circle (0.5);
		\end{tikzpicture}
	\end{center}
and for $i=1, 2, 3$, $\tilde{Q}^{(i)}$ is a connected full subquiver of $\widehat{Q}$ such that $\tilde{Q}_0^{(i)} \cap \tilde{Q}_0^{(j)}=\emptyset$ if $i\not=j$.
	\begin{itemize}
		\item[(a)] If $\tilde{Q}^{(1)}$ has an infinite number of vertices, then  $s(\beta)^+=\{\beta\}$, $s(\beta)^-=\emptyset$, $t(\gamma)^-=\{\gamma\}$ and $t(\gamma)^+=\emptyset$. 
		\item[(b)] If $\tilde{Q}^{(3)}$ has an infinite number of vertices, then  $s(\alpha)^+=\{\alpha\}$, $s(\alpha)^-=\emptyset$, $s(\beta)^+=\{\beta\}$ and $s(\beta)^-=\emptyset$.
	\end{itemize}
\end{lemma}

\begin{proof} We will show only the item (a), the item (b) is similar.

Since $I$ is generated by monomial relations, then $\widehat{Q}$ is a connected and locally finite tree. Moreover, since $I$ is an admissible ideal of $K Q$ and $Q$ is finite, then there is  an integer $m \geq 2$ such that $\mathcal{R}_{\widehat{Q}}^m \subset \widehat{I}$.
	
If $s(\beta)^+\not=\{\beta\}$ or $s(\beta)^-\not=\emptyset$, then since $\widehat{Q}_0^{(1)}$ is an infinite set, it follows from Lemma \ref{lema} that $K\tilde{Q}/\tilde{I}$ has a full subcategory $\tilde{A}$ such that $\tilde{A}$ is isomorphic to path algebra of some of the following quivers with relations:
	\[
	\begin{array}{l}
	\begin{tikzpicture}
	\node  (i) at (-2, 0) {(1)};
	\node  (o) at (0.5, -1.5) {and $J = \mathcal{R}_{\Delta}^2$;};
	\node  (0) at (-1, 0) {$\Delta$:};
	\node  (b) at (0, 0) {$\bullet$};
	\node  (c) at (1, 0) {$\bullet$};
	\node  (d) at (2, 0) {$\bullet$};
	\node  (e) at (3, 0) {$\bullet$};
	\node  (f) at (4, 0) {$\bullet$};
	\node  (g) at (5, 0) {$\bullet$};
	\node  (h) at (6, 0) {$\bullet$};
	\node  (i) at (7, 0) {$\bullet$};
	\node  (j) at (8, 0) {$\bullet$};
	\node  (k) at (6, -1) {$\bullet$};
	\draw [-] (b) -- node[above] { $\omega_{5}$}  (c);
	\draw [-] (c) -- node[above] { $\omega_{4}$} (d);
	\draw [-] (d) -- node[above] { $\omega_{3}$} (e);
	\draw [-] (e) -- node[above] { $\omega_{2}$} (f);
	\draw [-] (f) -- node[above] { $\omega_{1}$} (g);
	\draw [->] (g) -- node[above] { $\omega_0$} (h);
	\draw [->] (i) -- node[above] { $\beta$} (h);
	\draw [->] (i) -- node[above] { $\lambda$} (j);
	\draw [->] (h) -- node[right] { $\gamma$} (k);
	\end{tikzpicture}
	\\
	
	\begin{tikzpicture}
	\node  (i) at (-2, 0) {(2)};
	\node  (o) at (0.5, -1.5) {and $H = \mathcal{R}_{\Omega}^2$;};
	\node  (0) at (-1, 0) {$\Omega$:};
	\node  (b) at (0, 0) {$\bullet$};
	\node  (c) at (1, 0) {$\bullet$};
	\node  (d) at (2, 0) {$\bullet$};
	\node  (e) at (3, 0) {$\bullet$};
	\node  (f) at (4, 0) {$\bullet$};
	\node  (g) at (5, 0) {$\bullet$};
	\node  (h) at (6, 0) {$\bullet$};
	\node  (i) at (7, 0) {$\bullet$};
	\node  (j) at (8, 0) {$\bullet$};
	\node  (k) at (6, -1) {$\bullet$};
	\draw [-] (b) -- node[above] { $\omega_{5}$}  (c);
	\draw [-] (c) -- node[above] { $\omega_{4}$} (d);
	\draw [-] (d) -- node[above] { $\omega_{3}$} (e);
	\draw [-] (e) -- node[above] { $\omega_{2}$} (f);
	\draw [-] (f) -- node[above] { $\omega_{1}$} (g);
	\draw [->] (g) -- node[above] { $\omega_0$} (h);
	\draw [->] (i) -- node[above] { $\beta$} (h);
	\draw [->] (j) -- node[above] { $\lambda$} (i);
	\draw [->] (h) -- node[right] { $\gamma$} (k);
	\end{tikzpicture}
	\\
	
	\begin{tikzpicture}
	\node  (i) at (-2, 0) {(3)};
	\node  (o) at (1, -1.5) {and $L = \langle  \Upsilon_2 \backslash \{\lambda\beta \} \rangle$.};
	\node  (0) at (-1, 0) {$\Upsilon$:};
	\node  (b) at (0, 0) {$\bullet$};
	\node  (c) at (1, 0) {$\bullet$};
	\node  (d) at (2, 0) {$\bullet$};
	\node  (e) at (3, 0) {$\bullet$};
	\node  (f) at (4, 0) {$\bullet$};
	\node  (g) at (5, 0) {$\bullet$};
	\node  (h) at (6, 0) {$\bullet$};
	\node  (i) at (7, 0) {$\bullet$};
	\node  (j) at (8, 0) {$\bullet$};
	\node  (k) at (6, -1) {$\bullet$};
	\draw [-] (b) -- node[above] { $\omega_{5}$}  (c);
	\draw [-] (c) -- node[above] { $\omega_{4}$} (d);
	\draw [-] (d) -- node[above] { $\omega_{3}$} (e);
	\draw [-] (e) -- node[above] { $\omega_{2}$} (f);
	\draw [-] (f) -- node[above] { $\omega_{1}$} (g);
	\draw [->] (g) -- node[above] { $\omega_0$} (h);
	\draw [->] (i) -- node[above] { $\beta$} (h);
	\draw [->] (j) -- node[above] { $\lambda$} (i);
	\draw [->] (h) -- node[right] { $\gamma$} (k);
	\end{tikzpicture}
	\end{array}
	\]  
In the cases (1) and (2), we have that $A$ is derived wild by Lemma \ref{lemma1}, a contradiction. If $\tilde{A} \cong K \Upsilon / L$,  then since $s(\lambda) \in \Upsilon_0$ is a source and there is no relation $\rho \in \Upsilon_2\backslash \{\lambda\beta \}$ such that $s(\rho)=s(\lambda)$, then $\tilde{A}$ is derived equivalent to  $\mu_{t(\lambda)}^+(K \Upsilon / L) \cong K \Delta/ J$. Hence $A$ is derived wild by Lemma \ref{lemma1}, a contradiction. Therefore $s(\beta)^+ = \{\beta\}$ and $s(\beta)^-=\emptyset$. 
	
If $t(\gamma)^-\not=\{\gamma\}$ or $t(\gamma)^+\not= \emptyset$, then since $\widehat{Q}_0^{(1)}$ is an infinite set, it follows from Lemma \ref{lema} that $K\widehat{Q}/\widehat{I}$ has a full subcategory $\tilde{B}$ which is isomorphic to path algebra of some of the following quivers with relations:
	\[
	\begin{array}{l}
	\begin{tikzpicture}
	\node  (i) at (-2, 0) {(1')};
	\node  (o) at (0.5, -1.5) {and $J' = \mathcal{R}_{\Delta'}^2$;};
	\node  (0) at (-1, 0) {$\Delta'$:};
	\node  (b) at (0, 0) {$\bullet$};
	\node  (c) at (1, 0) {$\bullet$};
	\node  (d) at (2, 0) {$\bullet$};
	\node  (e) at (3, 0) {$\bullet$};
	\node  (f) at (4, 0) {$\bullet$};
	\node  (g) at (5, 0) {$\bullet$};
	\node  (h) at (6, 0) {$\bullet$};
	\node  (i) at (7, 0) {$\bullet$};
	\node  (j) at (8, 0) {$\bullet$};
	\node  (k) at (6, -1) {$\bullet$};
	\draw [-] (b) -- node[above] { $\omega_{5}$}  (c);
	\draw [-] (c) -- node[above] { $\omega_{4}$} (d);
	\draw [-] (d) -- node[above] { $\omega_{3}$} (e);
	\draw [-] (e) -- node[above] { $\omega_{2}$} (f);
	\draw [-] (f) -- node[above] { $\omega_{1}$} (g);
	\draw [->] (g) -- node[above] { $\omega_0$} (h);
	\draw [->] (h) -- node[above] { $\gamma$} (i);
	\draw [->] (j) -- node[above] { $\lambda$} (i);
	\draw [->] (k) -- node[right] { $\beta$} (h);
	\end{tikzpicture}
	\\
	
	\begin{tikzpicture}
	\node  (i) at (-2, 0) {(2')};
	\node  (o) at (0.5, -1.5) {and $H' = \mathcal{R}_{\Omega'}^2$;};
	\node  (0) at (-1, 0) {$\Omega'$:};
	\node  (b) at (0, 0) {$\bullet$};
	\node  (c) at (1, 0) {$\bullet$};
	\node  (d) at (2, 0) {$\bullet$};
	\node  (e) at (3, 0) {$\bullet$};
	\node  (f) at (4, 0) {$\bullet$};
	\node  (g) at (5, 0) {$\bullet$};
	\node  (h) at (6, 0) {$\bullet$};
	\node  (i) at (7, 0) {$\bullet$};
	\node  (j) at (8, 0) {$\bullet$};
	\node  (k) at (6, -1) {$\bullet$};
	\draw [-] (b) -- node[above] { $\omega_{5}$}  (c);
	\draw [-] (c) -- node[above] { $\omega_{4}$} (d);
	\draw [-] (d) -- node[above] { $\omega_{3}$} (e);
	\draw [-] (e) -- node[above] { $\omega_{2}$} (f);
	\draw [-] (f) -- node[above] { $\omega_{1}$} (g);
	\draw [->] (g) -- node[above] { $\omega_0$} (h);
	\draw [->] (h) -- node[above] { $\gamma$} (i);
	\draw [->] (i) -- node[above] { $\lambda$} (j);
	\draw [->] (k) -- node[right] { $\beta$} (h);
	\end{tikzpicture}
	\\
	\begin{tikzpicture}
	\node  (i) at (-2, 0) {(3')};
	\node  (o) at (1, -1.5) {and $L' = \langle   \Upsilon'_2 \backslash\{\gamma\lambda\} \rangle$.};
	\node  (0) at (-1, 0) {$\Upsilon'$:};
	\node  (b) at (0, 0) {$\bullet$};
	\node  (c) at (1, 0) {$\bullet$};
	\node  (d) at (2, 0) {$\bullet$};
	\node  (e) at (3, 0) {$\bullet$};
	\node  (f) at (4, 0) {$\bullet$};
	\node  (g) at (5, 0) {$\bullet$};
	\node  (h) at (6, 0) {$\bullet$};
	\node  (i) at (7, 0) {$\bullet$};
	\node  (j) at (8, 0) {$\bullet$};
	\node  (k) at (6, -1) {$\bullet$};
	\draw [-] (b) -- node[above] { $\omega_{5}$}  (c);
	\draw [-] (c) -- node[above] { $\omega_{4}$} (d);
	\draw [-] (d) -- node[above] { $\omega_{3}$} (e);
	\draw [-] (e) -- node[above] { $\omega_{2}$} (f);
	\draw [-] (f) -- node[above] { $\omega_{1}$} (g);
	\draw [->] (g) -- node[above] { $\omega_0$} (h);
	\draw [->] (h) -- node[above] { $\gamma$} (i);
	\draw [->] (i) -- node[above] { $\lambda$} (j);
	\draw [->] (k) -- node[right] { $\beta$} (h);
	\end{tikzpicture}
	\end{array}
	\]  
Again, this implies that $A$ is derived wild, a contradiction. Therefore  $t(\gamma)^-=\{\gamma\}$ and $t(\gamma)^+= \emptyset$.
\end{proof}

\begin{lemma}\label{lemma9}
Let $A=KQ/I$ be a derived tame quadratic string algebra where $Q$ is not a tree. Suppose that there are $\alpha, \beta, \gamma, \delta \in Q_1$ and $x \in Q_0$ such that $t(\alpha)= t(\beta)=s(\gamma)=s(\delta)=x$.
\begin{itemize}
	\item[(a)]If the vertices $s(\alpha), s(\beta), x, t(\gamma)$ and $t(\delta)$ are pairwise distinct,\linebreak $\alpha\gamma, \beta\gamma, \alpha\delta \in I$ but $\beta\delta \notin I$, then $s(\alpha)^+=\{\alpha\} $, $s(\alpha)^-=\emptyset$, $t(\gamma)^+=\emptyset$ and $t(\gamma)^-=\{\gamma\}$. In particular, $x \in E_1$.
	\item [(b)]If the vertices $s(\alpha), s(\beta), x$ and $t(\gamma)$ are pairwise distinct, $s(\beta)=t(\delta)$, $\alpha\gamma, \beta\gamma, \alpha\delta, \delta\beta \in I$ but $\beta\delta \notin I$, then $s(\alpha)^+=\{\alpha\} $, $s(\alpha)^-=\emptyset$, $t(\gamma)^+=\emptyset$ and $t(\gamma)^-=\{\gamma\}$. In particular, $x \in E_2$.
	\end{itemize}
\end{lemma}

\begin{proof} We only prove (a), the proof of (b) is similar.
	 
Let $K\widehat{Q}/\widehat{I}$ be the Galois covering of $A$. Since $Q$ has cycles and $I$ is generated by monomial relations, then $\widehat{Q}$ is a connected and locally finite tree with infinite vertices. Moreover, since $Q$ is a finite quiver and $I$ is an admissible ideal of $KQ$, then there is $m \geq 2$ such that $\mathcal{R}_Q^m \subset I$, hence $\mathcal{R}_{\widehat{Q}}^m \subset \widehat{I}$. 
		
By construction, $\widehat{Q}$ has the following shape
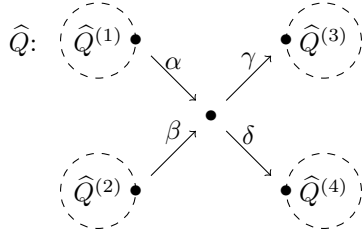
\begin{center}
	\begin{tikzpicture}
	\node  (x) at (-1.5, 1) {$\widehat{Q}$:};
	\node  (a) at (0, 1) {$\bullet$};
	\node  (b) at (0, -1) {$\bullet$};
	\node  (c) at (1, 0) {$\bullet$};
	\node  (d) at (2, 1) {$\bullet$};
	\node  (e) at (2, -1) {$\bullet$};
	\draw [->] (a) -- node[above] { $\alpha$}  (c);
	\draw [->] (b) -- node[above] { $\beta$} (c) ;
	\draw [->] (c) -- node[above] { $\gamma$} (d);
	\draw [->] (c) -- node[above] { $\delta$} (e);
	\draw[dashed] (-0.5,1) node {$\widehat{Q}^{(1)}$} (-0.5,1) circle (0.5);
	\draw[dashed] (-0.5,-1) node {$\widehat{Q}^{(2)}$} (-0.5,-1) circle (0.5);
	\draw[dashed] (2.5,1) node {$\widehat{Q}^{(3)}$} (2.5,1) circle (0.5);
	\draw[dashed] (2.5,-1) node {$\widehat{Q}^{(4)}$} (2.5,-1) circle (0.5);
\end{tikzpicture}
\end{center}
where for $1\leq i \leq 4$, $\widehat{Q}^{(i)}$ is a full and connected  subquiver of $\widehat{Q}$ such that \linebreak $\widehat{Q}_0^{(i)} \cap \widehat{Q}_0^{(j)}=\emptyset$ if $i\not=j$.  Moreover, since $\widehat{Q}$ has an infinite number of vertices, at last one of the $\widehat{Q}^{(i)}$'s must have an infinite many of vertices. 
		
If $\widehat{Q}_0^{(1)}$ is an infinite set, then by Lemma \ref{lema} we can choice vertices of $\widehat{Q}$ such that they define a full subcategory $\tilde{A}=K\tilde{Q}/\tilde{I}$ where 
		
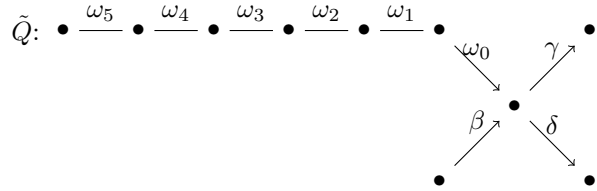
\begin{center}
			\begin{tikzpicture}
			\node  (q) at (-0.5, 0) {$\tilde{Q}$:};
			\node  (0) at (0, 0) {$\bullet$};
			\node  (1) at (1, 0) {$\bullet$};
			\node  (2) at (2, 0) {$\bullet$};
			\node  (3) at (3, 0) {$\bullet$};
			\node  (4) at (4, 0) {$\bullet$};
			\node  (5) at (5, 0) {$\bullet$};
			\node  (6) at (5, -2) {$\bullet$};
			\node  (7) at (6, -1) {$\bullet$};
			\node  (8) at (7, 0) {$\bullet$};
			\node  (9) at (7, -2) {$\bullet$};
			\draw [-] (0) -- node[above] { $\omega_5$} (1);
			\draw [-] (1) -- node[above] { $\omega_4$} (2);
			\draw [-] (2) -- node[above] { $\omega_3$} (3);
			\draw [-] (3) -- node[above] { $\omega_2$} (4);
			\draw [-] (4) -- node[above] { $\omega_1$} (5);
			\draw [->] (5) -- node[above] { $\omega_0$} (7);
			\draw [->] (6) -- node[above] { $\beta$} (7);
			\draw [->] (7) -- node[above] { $\gamma$} (8);
			\draw [->] (7) -- node[above] { $\delta$} (9);
			\end{tikzpicture}
		\end{center}
and $\tilde{I}=\langle\tilde{Q}_2  \backslash \{\beta\delta\} \rangle$. By Lemma \ref{lemma8} we get that $K\tilde{Q}/\tilde{I}$ is derived wild, thus $A$ is derived wild  by Lemma \ref{lemma1}, a contradiction. Hence $\widehat{Q}_0^{(1)}$ is a finite set. Moreover, the same argument applied to $A^{op}$ shows that $\widehat{Q}_0^{(3)}$ is a finite set. Therefore $\widehat{Q}_0^{(2)}$ or $\widehat{Q}_0^{(4)}$ must be an infinite set.
		
If $\widehat{Q}_0^{(2)}$ is an infinite set, then $s(\alpha)^+=\{\alpha\}$, $s(\alpha)^-= \emptyset$, $t(\gamma)^-=\{\gamma\}$ and $t(\gamma)^+= \emptyset$ by Lemma \ref{lemma13}(a). Dually, if $\widehat{Q}_0^{(4)}$ is an infinite set, then $s(\alpha)^+=\{\alpha\}$, $s(\alpha)^-= \emptyset$, $t(\gamma)^-=\{\gamma\}$ and $t(\gamma)^+= \emptyset$.  
\end{proof}

\begin{lemma}\label{lemma10}
Let $A=KQ/I$ be a derived tame quadratic string algebra where $Q$ is not a tree. Suppose that there are $\alpha, \beta, \gamma  \in Q_1$ and $x \in Q_0$ such that the vertices $s(\alpha), s(\beta), x$ and $ t(\gamma)$ are pairwise distinct, $\alpha\gamma, \beta\gamma \in I$, $x^+=\{\gamma\}$ and $t(\alpha)= t(\beta)=s(\gamma)=x$.
\begin{itemize}
\item[(a)] If $t(\gamma)^-\not=\{\gamma\}$ or $t(\gamma)^+=\emptyset$, then $s(\alpha)^+=\{\alpha\}$, $s(\alpha)^-=\emptyset$, $s(\beta)^+=\{\beta\}$ and $s(\beta)^-=\emptyset$. In particular, $x \in E_3$.	
\item[(b)] If $s(\alpha)^+ \not=\{\alpha\}$ or $s(\alpha)^-=\emptyset$, then $s(\beta)^+=\{\beta\}$, $s(\beta)^-=\emptyset$, $t(\gamma)^+=\emptyset$ and $t(\gamma)^-=\{\gamma\}$. In particular, $x \in E_4$.
\end{itemize}
\end{lemma}
\begin{proof}
Let $K\widehat{Q}/\widehat{I}$ be the Galois covering of $A$. Since $Q$ has cycles and $I$ is generated by monimial relations, then $\widehat{Q}$ is a connected and locally finite tree with an infinite number of vertices. Moreover, since $Q$ is a finite quiver and $I$ is an admissible ideal of $KQ$, then there is $m \geq 2$ such that $\mathcal{R}_Q^m \subset I$, hence $\mathcal{R}_{\widehat{Q}}^m \subset \widehat{I}$. 

By construction, $\widehat{Q}$ has the following shape
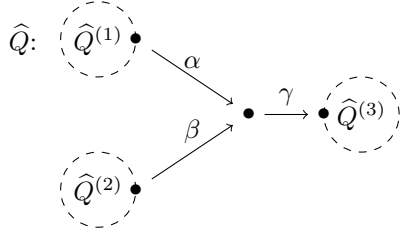
\begin{center}
	\begin{tikzpicture}
	\node  (x) at (1, 1) {$\widehat{Q}$:};
	\node  (a) at (2.5, 1) {$\bullet$};
	\node  (b) at (2.5, -1) {$\bullet$};
	\node  (c) at (4, 0) {$\bullet$};
	\node  (d) at (5, 0) {$\bullet$};
	\draw [->] (a) -- node[above] { $\alpha$}  (c);
	\draw [->] (b) -- node[above] { $\beta$} (c) ;
	\draw [->] (c) -- node[above] { $\gamma$} (d);
	\draw[dashed] (2,-1) node {$\widehat{Q}^{(2)}$} (2,-1) circle (0.5);
	\draw[dashed] (2,1) node {$\widehat{Q}^{(1)}$} (2,1) circle (0.5);
	\draw[dashed] (5.5,0) node {$\widehat{Q}^{(3)}$} (5.5,0) circle (0.5);
	\end{tikzpicture}
\end{center}
where for $i=1, 2, 3$, $\widehat{Q}^{(i)}$ is a full and connected subquiver of $\widehat{Q}$ such that \linebreak $\widehat{Q}_0^{(i)} \cap \widehat{Q}_0^{(j)}=\emptyset$ if $i\not=j$.  Moreover, since $\widehat{Q}$ has an infinite many of vertices, then at last one of the  $\widehat{Q}^{(i)}$'s must have an infinite many of vertices. 

If $t(\gamma)^-\not=\{\gamma\}$ or $t(\gamma)^+\not=\emptyset$, then $\widehat{Q}_0^{(1)}$ and $\widehat{Q}_0^{(2)}$ are finite sets by Lemma \ref{lemma13} (a). Hence  $\widehat{Q}_0^{(3)}$ is an infinite set, and so $s(\alpha)^+=\{\alpha\}$, $s(\alpha)^-=\emptyset$, $s(\beta)^+=\{\beta\}$ and $s(\beta)^-=\emptyset$ by Lemma \ref{lemma13} (b). Thus the item (a) follows.

If $s(\alpha)^+ \not=\{\alpha\}$ or $s(\alpha)^-\not=\emptyset$, then   $\widehat{Q}_0^{(2)}$ and $\widehat{Q}_0^{(3)}$ are finite sets by Lemma \ref{lemma13}. Hence $\widehat{Q}_0^{(1)}$ is an infinite set, and so $s(\beta)^+=\{\beta\}$, $s(\beta)^-=\emptyset$, $t(\gamma)^+=\emptyset$ and $t(\gamma)^-=\{\gamma\}$ by Lemma \ref{lemma13} (a). Thus, the item (b) is done.
\end{proof} 

\begin{lemma}\label{lemma11}
Let $A=KQ/I$ be a derived tame quadratic string algebra where $Q$ is not a tree. Suppose that there are $\alpha, \gamma, \delta  \in Q_1$  and $x \in Q_0$ such that $t(\alpha)= s(\delta)=s(\gamma)=x$, the vertices $s(\alpha), x, t(\delta)$ and $ t(\gamma)$ are pairwise distinct, $\alpha\gamma, \alpha\delta \in I$ and  $x^-=\{\alpha\}$.
\begin{itemize}
	\item[(a)] If $s(\alpha)^+ \not=\{\alpha\}$ or $s(\alpha)^-=\emptyset$, then $t(\delta)^+=\emptyset$, $t(\delta)^-=\{\delta\}$, $t(\gamma)^+=\emptyset$ and $t(\gamma)^-=\{\gamma\}$. In particular, $x \in E_5$.
	\item[(b)] If $t(\gamma)^-\not=\{\gamma\}$ or $t(\gamma)^+\not=\emptyset$, then $s(\alpha)^+=\{\alpha\}$, $s(\alpha)^-=\emptyset$, $t(\delta)^+=\emptyset$ and $t(\delta)^+=\{\delta\}$. In particular, $x \in E_6$.
\end{itemize}
\end{lemma}
\begin{proof}
It is sufficient to apply the  Lemma \ref{lemma10} to $A^{op}$.
\end{proof}

Now we can prove the Theorems B and D. \\

\noindent\textbf{Proof of Theorem B.} If $A$ is a gqs algebra, then $A$ is derived tame by Theorem A. 

In the other hand, if $A$ is a derived tame string quadratic algebra and let $x$ be the non gentle vertex of $A$, then one of the conditions $(1)$, $(2)$, $(3)$ or $(4)$ described in Proposition \ref{prop2} holds. 

If $x$ satisfy the condition (1), then there is arrows $\alpha, \beta, \gamma \in Q_1$ such that $t(\alpha)=t(\beta)=s(\gamma)=x$, $x^+=\{x\}$, $s(\alpha), s(\beta), x$ and $t(\gamma)$ are pairwise distinct, and $\alpha\gamma, \beta\gamma \in I$. Since $Q$ is connected and it has cycles, then there is an arrow $\lambda \in Q_1\backslash \{\alpha,\beta,\gamma\}$ such that $s(\lambda) \in \{s(\alpha), s(\beta),t(\gamma)\}$   or $t(\lambda) \in \{s(\alpha), s(\beta),t(\gamma)\}$. Thus $x$ belongs to $E_3$ or $E_4$ by Lemma \ref{lemma10}.

If $x$ satisfy the condition (2), we have a dual situation to the condition (1). Hence, it follows from Lemma \ref{lemma11} that $x$ belongs to $E_5$ or $E_6$. 

Finally, if $x$ satisfy the condition (3) or (4), then follows from Lemma \ref{lemma9} that $x$ belongs to $E_1$ or $E_2$, respectively. 

Therefore, we have that every non gentle vertex of $A$ is an exceptional vertex. Hence $A$ is a gqs algebra.\hspace*{8,5cm} $\square$ \\   

\noindent\textbf{Proof of Theorem D.} Let $A$ be a quadratic string algebra.

Suppose that $A$ is derived tame. If $A$ is a tree, then it follows from Theorem 1.1 of \cite{MR1868357} that its Euler form $\chi_A$ is non negative. But, if $A$ is not a tree, then $A$ is gqs algebra by Theorem B.

Conversely, if $A$ is tree and $\chi_A$ is non negative, then  $A$ is derived tame by Theorem 1.1 of \cite{MR1868357}. Moreover, if $A$ is gqs algebra, then  $A$ is derived tame by Theorem A.  \hspace*{12.3cm} $\square$

\section*{Acknowledgments:} The author would like to thank Viktor Bekkert for helpful discussions.
\bibliographystyle{amsplain}
\bibliography{paper}

\end{document}